\documentclass{amsart}
\usepackage{todonotes}
\usepackage{amsmath}
\usepackage{amsfonts}
\usepackage{amssymb}
\usepackage[all]{xy}
\usepackage{hyperref}
\usepackage{tikz}
\usepackage{tikz-cd}
\usetikzlibrary{calc}

\definecolor{shaded}{rgb}{0.99, 0.76, 0.8} % pinkish
\theoremstyle{plain}
\newtheorem*{introtheorem}{Theorem}
\newtheorem{theorem}{Theorem}[section]
\newtheorem{proposition}[theorem]{Proposition}
\newtheorem{lemma}[theorem]{Lemma}
\newtheorem{corollary}[theorem]{Corollary}

\newtheorem*{proposition*}{Proposition}

\theoremstyle{definition}
\newtheorem{definition}[theorem]{Definition}
\newtheorem{example}[theorem]{Example}

\theoremstyle{remark}
\newtheorem{remark}[theorem]{Remark}

\newcommand{\secref}[1]{Section~\ref{#1}}
\newcommand{\thmref}[1]{Theorem~\ref{#1}}
\newcommand{\propref}[1]{Proposition~\ref{#1}}
\newcommand{\lemref}[1]{Lemma~\ref{#1}}
\newcommand{\corref}[1]{Corollary~\ref{#1}}

\newcommand{\exref}[1]{Example~\ref{#1}}

\newcommand{\remref}[1]{Remark~\ref{#1}}

\newcommand{\defref}[1]{Definition~\ref{#1}}

\newcommand{\Z}{\mathbb{Z}}

\renewcommand{\bar}{\overline}

\begin{document}

\title[Simplicial Loop Space]{The Simplicial Loop Space of a Simplicial Complex}

\author[G.\ Lupton]{Gregory Lupton}
\author[J.\ Scott]{Jonathan Scott}

\address{Department of Mathematics and Statistics, Cleveland State University, Cleveland OH 44115 U.S.A.}

\email{g.lupton@csuohio.edu}

\email{j.a.scott3@csuohio.edu}

\date{\today}

\keywords{Simplicial Complex, Edge Group, Fundamental Group, Based Loop Space, Face Group, Second Homotopy Group, Spatial Realization}
\subjclass[2020]{ (Primary) 55U10;  (Secondary) 55P35 05E45}

\begin{abstract}   Given a simplicial complex $X$, we construct a simplicial complex $\Omega X$ that may be regarded as a combinatorial version of the based loop space of a topological space. Our construction explicitly describes the simplices of $\Omega X$ directly in terms of the simplices of $X$.  Working at a purely combinatorial level, we show two main results that confirm the (combinatorial) algebraic topology of our $\Omega X$ behaves like that of the topological based loop space. Whereas our $\Omega X$ is generally a disconnected simplical complex, each component of $\Omega X$ has the same edge group, up to isomorphism (\thmref{thm: homogeneous components}).  We show an isomorphism between the edge group of $\Omega X$ and the combinatorial second homotopy group of $X$ as it has been defined in separate work (\thmref{thm: edge Omega X and face X}). Finally, we enter the topological setting and, relying on prior work of Stone, show a homotopy equivalence between the spatial realization of our $\Omega X$ and the based loop space of the spatial realization of $X$ (\thmref{thm:r-equiv}).  
\end{abstract}

\maketitle

\section{Introduction}\label{sec: Intro}

The \emph{edge group} of a simplicial complex is a well-known construction that gives a combinatorial description of the fundamental group (see \cite{Mau96} for example). If $E(X)$ denotes the edge group of the simplicial complex $X$, then we have an isomorphism of groups $E(X) \cong \pi_1(|X|)$ between the (combinatorially defined) edge group of $X$ and the (topological) fundamental group of $|X|$, the spatial realization of $X$. Now \cite{L-S-S} gives a like-minded combinatorial description of the second homotopy group of a simplicial complex.  In that work, this group is called the \emph{face group} of a simplicial complex, it is denoted by $F(X)$, and we show an isomorphism of groups $F(X) \cong \pi_2(|X|)$ between the (combinatorially defined) face group of $X$ and the (topological) second homotopy group of $|X|$. In the topological setting, for a topological space $Y$, we have an isomorphism of groups 
$\pi_2(Y) \cong \pi_1(\Omega Y)$, where $\Omega Y$ denotes the based loop space of $Y$.  It is natural to ask for a combinatorial version of this isomorphism.  We establish such in this paper by describing in a very concrete way  a simplicial complex $\Omega X$ associated with each simplicial complex $X$ that plays the role of a combinatorial based loop space in a most satisfactory way.  

\begin{introtheorem}
    Let $X$ be a simplicial complex.  There is an associated simplicial complex $\Omega X$ that satisfies
    $$F(X) \cong E(\Omega X).$$
\end{introtheorem}

We describe the simplicial complex $\Omega X$ in \secref{sec: Omega X} and show the isomorphism above in \thmref{thm: edge Omega X and face X}.  Other results support the claim of $\Omega X$ to be considered the simplicial loop space. In the topological setting, the connected components of the based loop space $\Omega Y$ are in one-to-one correspondence with the elements of the fundamental group $\pi_1(Y)$. Also, it is well-known that the connected components of $\Omega Y$ are ``homogeneous" in the sense that they are homotopy equivalent to each other.  Here, we show likewise that the (edge-path-) connected components of $\Omega X$ are in one-to-one correspondence with the elements of the edge group $E(X)$ (\corref{cor: components and edge group}).  Whilst we are not able to show the components of $\Omega X$ are combinatorially equivalent in some sense, we do show that different components  have isomorphic edge group (\thmref{thm: homogeneous components}). All results mentioned thus far are shown using entirely combinatorial methods, based on our explicit description of $\Omega X$ as a simplicial complex.

Following these combinatorial results, we pass to the topological setting and show that $|\Omega X|$, the spatial realization of our simplicial loop space, is homotopy equivalent to $\Omega |X|$, the (topological) based loop space of the spatial realization of $X$. Whilst our earlier results may be deduced from this homotopy equivalence, we believe the combinatorial arguments have independent value as they confirm that our $\Omega X$ has the correct combinatorial structure to be considered as the simplicial loop space. It does not merely qualify at the level of spatial realization.  

As this work was being completed, we  became aware of the work of \cite{Gr02}.  There, Grandis associates to a simplicial complex $X$ various other simplicial complexes that model 
spaces of maps into $X$, such as path and loop spaces.  In particular, he defines the homotopy group $\widetilde{\pi_2}(X)$ of a simplicial complex $X$ as $\widetilde{\pi_1}(\widetilde{\Omega} X)$, where $\widetilde{\Omega} X$ is a simplicial complex that plays the role of the based loop space of $X$, and $\widetilde{\pi_1}(X)$ is an ``extrinsic" fundamental group that is isomorphic to the ```intrinsic" edge  group of $X$, just as we have here. We use the decorated notations $\widetilde{\pi_i}$ and $\widetilde{\Omega}$ here to distinguish those of \cite{Gr02} from ours in the following discussion. Whilst there are considerable similarities between our work here and some of that in \cite{Gr02} (which has a much wider-ranging scope than our work here), we point here to some of their differences. We also discuss some differences between our development and that of \cite{Gr02} in \cite{L-S-S}. First, our work here harks back to prior work in digital topology.  From that digital point of view,  it is useful  to have a development that matches the way in which we work in the digital setting: rectangular (finite) domains, left homotopy (using a cylinder object), trivial extensions. The trivial extensions we use here are the same as the ``delays" of \cite{Gr02}, but we have adopted their use from Boxer's work in digital topology \cite{Bo99}. We chose our paths and loops so as to correspond directly to edge paths in $X$, namely finite sequences of vertices. Second, Grandis shows an isomorphism  $\widetilde{\pi_2}(X) \cong \pi_2(|X|)$, so our $E(\Omega X, \mathbf{x_0})$, the edge group of $\Omega X$,  is apparently isomorphic to Grandis' 
$\widetilde{\pi_2}(X)$ (\emph{per} \thmref{thm: edge Omega X and face X}, \thmref{thm:r-equiv}, and  \cite[Th.8.1]{L-S-S}.  However, the isomorphism is not so evident without passing through the spatial realization. The proof of  \thmref{thm: edge Omega X and face X}  takes considerable work handling technical details that arise in what, overall, is a fairly clear line of argument. 
Third, our simplicial complex $\Omega X$ here is combinatorially different from the $\widetilde{\Omega}X$ of \cite{Gr02} in several ways.  A path in \cite{Gr02} has domain $\Z$. Maps such as our face spheres here (``double paths" in \cite{Gr02}, conceived of as paths in $\widetilde{\Omega}X$) have domain $\Z^2$.  The vertices of $\widetilde{\Omega}X$ and $\Omega X$ are different. For instance, all constant maps at a vertex are treated as the same map in \cite{Gr02}, whereas we distinguish between constant paths of different lengths. On the other hand, each  edge path in $X$ corresponds naturally to a unique vertex of $\Omega X$, whereas the same sequence of vertices may be used for infinitely many vertices of $\widetilde{\Omega}X$, including ones with negative ``support." More significantly, perhaps, our $\Omega X$ has vertices of finite valency, given $X$ finite. By contrast,   $\widetilde{\Omega}X$ has vertices of infinite valency (for finite $X$).   
This ``locally finite" property of $\Omega X$ could play a significant role in any algorithmic developments from our work, which is a direction we hope to pursue (see \secref{sec: future}).  
 Fourth, and finally, we show a homotopy equivalence between $|\Omega X|$ and the based loop space of $|X|$, a result that has no counterpart in \cite{Gr02}. 
 
The paper is organized as follows.   
Since our simplicial complex $\Omega X$ involves edge paths and edge loops in $X$, we begin by discussing the latter topics and reviewing the edge group in \secref{sec: edge loops}. We start the section with a review of some basic notation and terminology surrounding the notion of contiguity. Whilst we will define terms and review notation as we go along, we do rely on a basic familiarity with standard notions of simplicial complexes and standard ways of operating with them.  We describe the simplicial complex $\Omega X$ in \secref{sec: Omega X}. The simplices of $\Omega X$ are described explicitly and directly in terms of the simplices of $X$. We give a number of explicit examples of simplices of $\Omega X$, including some $2$- and $3$-simplices of $\Omega X$ that play a role in the sequel. In \secref{sec: loops in Omega X} we discuss edge loops in $\Omega X$. Here we show two results that demonstrate that the combinatorial properties of $\Omega X$ compare well with the topological properties of $\Omega Y$, the based loop space of a topological space $Y$. \propref{prop: left-right homotopy} shows that an (edge) path from one vertex to another in $\Omega X$ corresponds to the endpoints---considered as based loops in $X$---being equivalent. In \thmref{thm: homogeneous components} we show the components of $\Omega X$ are homogeneous with respect to the edge group. We discuss there the ways in which these results compare with their topological counterparts.  In \secref{sec: face group} we review the construction of the face group from \cite{L-S-S} and prepare the way for the main result of \secref{sec: F(X) iso E(Omega X)}. In \thmref{thm: edge Omega X and face X} we show the isomorphism of groups $F(X) \cong E(\Omega X)$. We turn to the topological setting for the remainder of the paper. In \secref{sec: spatial realization} we relate the spatial realization of our $\Omega X$ with the spatial realization of a polyhedral complex constructed by Stone in \cite{St79}. Actually the main result here (\propref{prop: homotopy equivalent skeleta}) identifies certain approximations to the full spatial realizations.  In the final   \secref{sec: limits}, we show these approximations form compatible direct systems and by passing to homotopy colimits we obtain the homotopy equivalence between the spatial realization of our $\Omega X$ and the based loop space of the spatial realization of $X$ (\thmref{thm:r-equiv}).         
The paper concludes with the brief \secref{sec: future},  which  indicates developments we intend to pursue in subsequent work.

\section{Edge Paths and Edge Loops in a Simplicial Complex; The Edge Group}\label{sec: edge loops}

For a general overview of  material on  simplicial complexes, we refer to \cite{Koz08}.  By a simplicial complex here, we mean an abstract simpicial complex.  
Let $(X, x_0)$ be a based simplicial complex ($x_0$ is some vertex of $X$).   
An \emph{edge path of length $m$} in $X$ is an ordered  sequence $( v_0, v_1, \ldots, v_m )$ of vertices of $X$ with each adjacent pair $\{ v_i, v_{i+1}\}$ an edge (a $1$-simplex) or a repeated vertex (a $0$-simplex) of $X$, for  $i=0, \ldots, m-1$.   The edge path is an \emph{edge loop} (of length $m$) if $v_0 = v_m$ and is a \emph{based edge loop} if $v_0 = v_m = x_0$.   A typical edge loop (or path) in $X$ will be denoted by $l$. A simplicial complex $X$ is (edge-path) \emph{connected} if, given any two vertices $v$ and $v'$ of $X$ there is some edge path $( v, v_1, \ldots, v_{m-1}, v' )$ in $X$. 
 Since we will be discussing based loops and homotopy groups, we assume $X$ is a connected simplicial complex (or a connected component of a simplicial complex).

\begin{definition}\label{def: contiguous edge loops}
Edge paths $l = ( v_0, \ldots, v_i, \ldots, v_m )$ and $l' = ( v'_0, \ldots, v'_i, \ldots, v'_m )$ in $X$ are \emph{contiguous} if $\{ v_i, v_{i+1}, v'_i, v'_{i+1} \}$ is a simplex of $X$, for each $0 \leq i \leq m-1$.  Note that some of the vertices involved may be repeats. We write $l \sim l'$ for the reflexive, symmetric, but not-necessarily transitive relation of contiguity. Note that contiguity of edge paths entails they are of the same length as each other.
\end{definition}

The term contiguity is usually applied to simplicial maps, namely a vertex map of simplicial complexes that sends simplices to simplices, as follows.

\begin{definition}\label{def: contiguous maps}
Suppose $K$ and $X$ are simplicial complexes. Two simplicial maps $f, g \colon K \to X$ are \emph{contiguous} if $f(\sigma) \cup g(\sigma)$ is some simplex of $X$, for each simplex $\sigma$ of $K$.   We write $f \sim g$ for the reflexive, symmetric, but not-necessarily transitive relation of contiguity. Note that it is sufficient to check this contiguity condition only for those $\sigma$ that are the maximal simplices of $K$.  
\end{definition}

Contiguity generates an equivalence relation on the set of simplicial maps between (fixed) simplicial complexes. Namely, two simplicial maps $f, f' \colon K \to X$ are \emph{contiguity equivalent} if there is a sequence of contiguities
$$f \sim f_1 \sim \cdots \sim f_{n-1} \sim f'$$
for simplicial maps $f_j\colon K \to X$. 
We write $f \simeq f'$ to denote that $f$ and $f'$ are contiguity equivalent maps.

Now an edge path of length $m$ may be viewed as a simplicial map  $l \colon I_m \to X$, where $I_m = [0, m]_{\mathbb Z}$, the simplical complex with vertices the integers from $0$ to $m$ inclusive and edges given by pairs of consecutive integers.  The edge paths $l$ and $l'$ of \defref{def: contiguous edge loops}  are contiguous according to that definition exactly when, considered as maps $l, l' \colon I_m \to X$, they are contiguous according to \defref{def: contiguous maps}.  More generally, we will say two edge paths $l$ and $l'$ as in \defref{def: contiguous edge loops} are \emph{contiguity equivalent}, and write $l \simeq l'$, if they are contiguity equivalent considered as simplicial maps $l, l' \colon I_m \to X$.    

In the sequel, we will use whichever of these two points of view on edge paths is most convenient---edge path as a sequence of vertices or edge path as a simplicial map.  We will feel free to switch from one point of view to the other without comment.

\subsection{The edge group of a simplicial complex}
For $(X, x_0)$ a based simplicial complex,  the \emph{edge group} is the set of equivalence classes of all based edge loops (all lengths) under the  equivalence relation generated by the two types of (reflexive, symmetric) move:

(i)  $( x_0, \ldots, v_i, \ldots, x_0 ) \approx ( x_0, \ldots, v_i, v_i, \ldots, x_0 )$ (i.e., repeat a vertex or delete a repeated vertex);

(ii)   $( x_0, \ldots, v_{i-1}, v_i, v_{i+1}, \ldots, x_0 ) \approx ( x_0, \ldots, v_{i-1}, v'_{i}, v_{i+1}, \ldots, x_0 )$ when  $\{ v_{i-1}, v_i, v'_i \}$ and $\{ v_i, v'_i,  v_{i+1} \}$ are both simplices of $X$.  That is, the (same length) edge loops are contiguous as we defined it above. 

We refer to this equivalence relation as \emph{extension-contiguity equivalence} of edge loops and denote it using ``$\approx$."  We denote by $[l]$ the extension-contiguity equivalence class of a based loop $l$ in $X$. It is straightforward to show that this set of equivalence classes form a group, the \emph{edge group of $X$} (based at $x_0$), which we denote by $E(X, x_0)$. It is well-known that the edge group is isomorphic to the fundamental group of the spatial realization: we have an isomorphism $E(X, x_0) \cong \pi_1(|X|, x_0)$. We refer to \cite{Mau96} or \cite{Spa81} for an exposition of the edge group.  Here, we recall some details of the development that we use in the sequel.

The product in the edge group is  induced by \emph{concatenation} of loops. Namely, for $l = ( x_0, v_1,  \ldots,  v_{m-1}, x_0 )$ and $l' = ( x_0, v'_1, \ldots, v'_{n-1}, x_0 )$, their product is the loop of length $m+n+1$
$$l\cdot l' = ( x_0, v_1,  \ldots,  v_{m-1}, x_0,  x_0, v'_1, \ldots, v'_{n-1}, x_0 ).$$
One checks that this induces a well-defined, associative product on extension-contiguity classes of loops. The role of a two-sided identity element is played by $[(x_0)]$, the extension-contiguity equivalence class of the trivial loop of length zero.  We will write $\mathbf{x_0}$ for $(x_0)$.  The \emph{reverse} of a loop
$l = ( x_0, v_1,  \ldots,  v_{m-1}, x_0 )$ is the loop $\widetilde{l} = ( x_0, v_{m-1},  \ldots,  v_{1}, x_0 )$. 

\begin{lemma}\label{lem: reverse as inverse}
Let $l = ( x_0, v_1,  \ldots,  v_{m-1}, x_0 )$ be a loop in $X$ of length $m$. We have a contiguity equivalence (of loops of length $2m+1$)
$$\mathbf{x_0^{2m+1}} \sim L^1 \sim \cdots \sim L^{m-1} \sim \widetilde{l} \cdot l,$$
where $\mathbf{x_0^{2m+1}} = (x_0, x_0, \ldots, x_0)$ denotes the constant loop at $x_0$ of length $2m+1$ and $L^i$ is the loop of length $2m+1$
$$( x_0, v_1,  \ldots, v_{i-1}, v_{i}, v_{i}, \ldots, v_{i},  v_{i}, v_{i-1}, \ldots v_{1}, x_0).$$
\end{lemma}

\begin{proof}
If we write $L^i_j$ for the $j$th vertex of loop $L^i$ and likewise for $L^{i+1}$, then $L^i_j = L^{i+1}_j=v_j$ for $0 \leq j \leq i$ and $L^i_j = L^{i+1}_j=v_{2m+1 - j}$ for $2m+1-i \leq j \leq 2m+1$. Outside these ranges, we have $L^i_j = v_i$ and $L^{i+1}_j=v_{i+1}$.  Then
$$
\{  L^i_j, L^i_{j+1}, L^{i+1}_j, L^{i+1}_{j+1} \} = 
\begin{cases} 
\{ v_j, v_{j+1} \} & 0 \leq j \leq i\\
\{ v_i, v_{i+1} \} & i+1 \leq j \leq 2m-i\\
\{ v_{2m+1 - j}, v_{2m - j} \} & 2m+1-i \leq j \leq 2m+1
\end{cases}
$$
and in all cases the contiguity condition (of \defref{def: contiguous edge loops})  is satisfied.
\end{proof}

From \lemref{lem: reverse as inverse}, it follows that in the edge group we have $[l]^{-1} = [ \widetilde{l}]$. We also recall here the invariance of $E(X, x_0)$ under change of basepoint (within a connected component, if $X$ is disconnected).

\begin{lemma}\label{lem: change of basepoint}
Let $\eta = ( x_0, v_1,  \ldots,  v_{m-1}, y_0 )$ be an edge path in $X$ from $x_0$ to $y_0$.  Then the map from loops in $X$ based at $y_0$ to loops in $X$ based at $x_0$ defined by
$$l \mapsto \eta\cdot l \cdot \widetilde{\eta}$$
induces an isomorphism of edge groups
$$\Phi_\eta \colon E(X, y_0) \to E(X, x_0).$$
\end{lemma}

\begin{proof}
The proof is straightforward.  We refer to \cite{Mau96} or \cite{Spa81} for details.    
\end{proof}

The final ingredient involving edge groups that we review here concerns induced homomorphisms.  A simplicial map $f \colon X \to Y$ induces a homomorphism of edge groups
$$f_* \colon E(X, x_0) \to E\left( Y, f(x_0) \right),$$
defined by setting $f_*([l]) := [ f\circ l]$.    We check that this gives a well-defined homomorphism.  The details are straightforward and we omit them.  The next result says that contiguity equivalent maps induce the same homomorphism of edge groups, up to a change of basis isomorphism. 

\begin{lemma}\label{lem: contiguous induced hom}
Suppose we have a contiguity equivalence of simplicial maps
$$f_0 \sim f_1 \sim \cdots \sim f_{n-1} \sim f_n \colon X \to Y.$$
Let $\eta = \left( f_n(x_0), f_{n-1}(x_0), \ldots, f_{1}(x_0), f_0(x_0) \right)$ be the edge path in $Y$ from $f_n(x_0)$ to $f_0(x_0)$ that follows the basepoint $x_0$ of $X$ through the contiguity equivalence (in reverse order).  Then we have a commutative diagram of homomorphisms of edge groups 
$$
\xymatrix{
E(X, x_0) \ar[r]^{(f_0)_*} \ar[rd]_{(f_n)_*}  & E\left( Y, f_0(x_0) \right) \ar[d]^{\Phi_\eta}_{\cong} \\
  &    E\left( Y, f_n(x_0) \right)
  }
$$
Here, $\Phi_\eta$ denotes the change of basepoint isomorphism as in \lemref{lem: change of basepoint}.  
\end{lemma}

\begin{proof}
First note that $\eta$ is indeed an edge path in $Y$, since the contiguity $f_i \sim f_{i+1}$ entails that $\{ f_i(x_0), f_{i+1}(x_0) \}$ is a simplex of $Y$ for the $0$-simplex $\{ x_0 \}$ of $X$, each $i = 0, \ldots, n-1$.  Consider first the single contiguity $f_0 \sim f_1 \colon X \to Y$.  Let $e_1 = \left(f_1(x_0), f_0(x_0)\right)$ denote the edge path of length $1$ from $f_1(x_0)$ to $f(x_0)$.  We check that the diagram
$$
\xymatrix{
E(X, x_0) \ar[r]^{(f_0)_*} \ar[rd]_{(f_1)_*}  & E\left( Y, f_0(x_0) \right) \ar[d]^{\Phi_{e_1}}_{\cong} \\
  &    E\left( Y, f_1(x_0) \right)
  }
$$
commutes.  For a typical element $[l] \in E(X, x_0)$, with $l = (x_0, v_1, \ldots, v_{m-1}, x_0)$, we have
$$
\begin{aligned}
\Phi_{e_1}\circ &(f_0)_*\left( [l] \right) = \left[ e_1 \cdot f_0(l) \cdot \widetilde{e_1} \right]\\
&=  \left[ \left( f_1(x_0), f_0(x_0), f_0(x_0), f_0(v_1), \ldots, f_0(v_{m-1}), f_0(x_0), f_0(x_0), f_1(x_0) \right) \right]\\
&=  \left[ \left( f_1(x_0), f_1(x_0), f_1(x_0), f_1(v_1), \ldots, f_1(v_{m-1}), f_1(x_0), f_1(x_0), f_1(x_0) \right) \right].
\end{aligned}
$$ 
The  contiguity $f_0 \sim f_1$ implies that $\{ f_0(v_i), f_0(v_{i+1}, f_1(v_i), f_1(v_{i+1} \}$ is a simplex of $Y$ for each $i$, which gives the last equality above. 
After removing repeats from either end of the representative loop displayed, we see that $\Phi_{e_1}\circ (f_0)_*\left( [l] \right) = (f_1)_*( [l] )$.

Iterating this step, we obtain a commutative diagram
$$
\xymatrix{
 & &  E\left( Y, f_0(x_0) \right) \ar[d]^{\Phi_{e_1}}_{\cong} \\
 & &  E\left( Y, f_1(x_0) \right) \ar[d]^{\Phi_{e_2}}_{\cong}\\
E(X, x_0) \ar[rruu]^{(f_0)_*} \ar[rru]^(0.65){(f_1)_*} \ar[rrd]_{(f_n)_*}  & \vdots & \vdots \ar[d]^{\Phi_{e_n}}_{\cong}\\
& &  E\left( Y, f_n(x_0) \right) 
  }
$$
Here, each $e_i = (f_i(x_0), f_{i-1}(x_0))$ denotes the edge path of length $1$ from $f_i(x_0)$ to $f_{i-1}(x_0)$, for $i=1, \ldots, n$. But the composition of the change of basis isomorphisms agrees with the single change of basis isomorphism $\Phi_\eta$ on edge groups.  To see this, write
$$
\begin{aligned}
\widehat{\eta} &=  (e_n\cdot e_{n-1} \cdot \,\cdots\, \cdot e_1) \\
&= (f_n(x_0), f_{n-1}(x_0), f_{n-1}(x_0), f_{n-2}(x_0), \ldots, f_2(x_0), f_1(x_0), f_1(x_0), f_0(x_0)).
\end{aligned}
$$ 
Then, for a typical $[l] \in E\left( Y, f_0(x_0)\right)$,  we have
$$\Phi_{e_n} \circ \cdots \circ \Phi_{e_1}([l]) = [\widehat{\eta} \cdot l \cdot \widetilde{ (\widehat{\eta})}] =   [\eta \cdot l \cdot \widetilde{ \eta}] = \Phi_\eta([l]),$$ 
with the second equality following by removing repeated vertices from $\widehat{\eta}$ and its reverse. The result follows.  
\end{proof}

\subsection{Trivial extensions of edge loops}
Just as we phrased contiguity in terms of simplicial maps, following \defref{def: contiguous maps} above, we may give a variant (but equivalent) description of the edge group in terms of simplicial maps. We develop this description in this subsection and the next. Contiguity equivalence of based loops in $X$ of a certain length, considered as simplicial maps $I_m \to X$, gives an operation on based loops that corresponds to a move of type  (i) from above. To operate with edge loops of different lengths, we need a device that is not commonly discussed in the context of simplicial maps, namely \emph{(trivial) extensions}.  

\begin{definition}\label{def: extensions}
For each $m$, define simplicial maps $\alpha_i \colon I_{m+1} \to I_m$ for each $i = 0, \ldots, m$ as follows:
$$\alpha_i(s) = 
\begin{cases}
s & 0 \leq s \leq i \\
 & \\
s-1 & i+1 \leq s \leq m+1 \\
\end{cases}
$$
Now for a loop or path $l\colon I_m \to X$ in $X$, we refer to a composition $l \circ \alpha_i \colon I_{m+1} \to X$ as an \emph{extension} of $l$. It is the loop or path obtained from $l$ by repeating the $i$th vertex.  More generally, if $I = \{ i_1, \ldots, i_r \}$ is a sequence with $0 \leq i_t \leq m+t-1$ for each $1 \leq t \leq r$, we write 
$$\alpha_I := \alpha_{i_1} \circ \alpha_{i_2} \circ \cdots \circ \alpha_{i_r} \colon I_{m+r} \to I_m$$
and also refer to $l\circ \alpha_I\colon I_{m+r} \to X$ as an extension of $l$. It is the loop obtained from $l$ by repeating the $i_r$th vertex, then repeating the $i_{r-1}$st vertex of that extended loop, and so-on. If $I = \{ m, \ldots, m \}$ ($r$-times), then we write $\alpha^r_m$ for $\alpha_I = \alpha_m \circ \cdots \circ \alpha_m$. Then $l\circ \alpha^r_m \colon I_{m+r} \to X$ is the loop or path obtained from $l$ by repeating $r$-times the final vertex of $l$ (the basepoint $x_0$ in case $l$ is a based loop in $X$). We distinguish this case by referring to $l\circ \alpha^r_m$ as a \emph{trivial extension} of (the based loop) $l \colon I_m \to X$.
\end{definition}

Whereas repeating different vertices of a loop or path generally leads to technically different extended loops or paths, the following result means that for most of our purposes it is sufficient just to extend trivially (by repeating just the final vertex).

\begin{proposition}\label{prop: contiguity results} 
Using the notation and vocabulary from around \defref{def: contiguous maps} and \defref{def: extensions}, we have the following.
    \begin{itemize}
\item[(a)] Given pairs of contiguous simplicial maps $f \sim f' \colon X \to Y$ and $g \sim g' \colon Y \to Z$, their compositions are contiguous: we have $g\circ f \sim g' \circ f' \colon X \to Z$.
\item[(b)] Given pairs of contiguity equivalent simplicial maps $f \simeq f' \colon X \to Y$ and $g \simeq g' \colon Y \to Z$, their compositions are contiguity equivalent: we have $g\circ f \simeq g' \circ f' \colon X \to Z$.
\item[(c)] We have a contiguity equivalence $\alpha_i \simeq \alpha_j \colon I_{m+1} \to I_m$ for any $0 \leq i < j \leq m$.
\item[(d)] Suppose $\alpha_I \colon I_{m+r} \to I_m$ and  $\alpha_J \colon I_{m+r} \to I_m$ are any maps of the form
$$\alpha_I := \alpha_{i_1} \circ \cdots \circ \alpha_{i_r} \qquad \mathrm{and} \qquad \alpha_J := \alpha_{j_1} \circ \cdots \circ \alpha_{j_r},$$
and $l \colon I_m \to X$ is any loop or path in $X$. Then we have a contiguity equivalence of extensions $l\circ \alpha_I \simeq l \circ \alpha_J \colon I_{m+r} \to I_m$. 
    \end{itemize}
\end{proposition}

\begin{proof}
(a) Suppose $\sigma$ is a simplex of $X$.  Because $f \sim f'$, we have $f(\sigma) \cup f'(\sigma)$ contained in some simplex $\sigma'$ of $Y$. Then we have $g\circ f(\sigma) \cup g'\circ f'(\sigma) \subseteq g(\sigma') \cup g'(\sigma')$, which is contained in some simplex $\sigma''$ of $Z$, as $g \sim g'$. Thus, we have $g\circ f \sim g' \circ f'$.  

(b) From part (a) we may write a contiguity equivalence of the form
$$g\circ f \sim g\circ f_1 \sim \cdots \sim g\circ f' \sim g_1\circ f' \sim \cdots \sim g'\circ f',$$
where $f \sim f_1 \sim \cdots \sim f'$ and $g \sim g_1 \sim \cdots \sim g'$ are contiguity equivalences $f \simeq f'$ and $g \simeq g'$.

(c) It is sufficient to show that we have $\alpha_i \sim \alpha_{i+1}$ for each $0 \leq i \leq m-1$. To this end, consider the typical simplex $\sigma = 
\{ s, s+1 \}$ of $I_{m+1}$, some  $0 \leq s \leq m$. Then a direct check gives that
$$\begin{aligned}
    \alpha_i(\sigma) \cup \alpha_{i+1}(\sigma) & = \left\{ \alpha_i(s), \alpha_i(s+1), \alpha_{i+1}(s), \alpha_{i+1}(s+1)\right\}\\
    &= \begin{cases}
        \{ s, s+1, s, s+1 \} = \{s, s+1\} & 0 \leq s \leq i-1 \\
        \{ i, i+1, i, i \} = \{i, i+1\} & s = i \\
         \{ i, i+1, i+1, i+1 \} = \{i, i+1\} & s = i+1 \\
         \{ s-1, s, s-1, s \} = \{s-1, s\} & i+2 \leq s \leq m. 
    \end{cases}
\end{aligned}$$
In all cases, we see that $\alpha_i(\sigma) \cup \alpha_{i+1}(\sigma)$ is a simplex of $I_m$. Namely, we have $\alpha_i \sim \alpha_{i+1}$.

(d) This follows directly from parts (b) and (c).
\end{proof}

\subsection{Extension contiguity of loops as maps; the edge group (bis)}

Combining contiguity with (trivial) extensions generates an equivalence relation on the set of all based loops considered as simplicial maps $l \colon I_{m} \to X$, for all lengths of interval $m \geq 0$, leading to the variant description of the edge group in terms of simplicial maps.   Specifically, we say that two simplicial maps $l \colon I_{m} \to X$ and $l' \colon I_{n} \to X$ are \emph{extension-contiguity equivalent} if there are extensions  $l\circ \alpha_{I}, l'\circ \alpha_J \colon I_{r} \to X$ (to some common length $r \geq m, n$) with $l\circ \alpha_{I}$ and $l'\circ \alpha_J$ contiguity equivalent.

The process by which one loop is deformed by an extension-contiguity equivalence generally consists of a sequence of extensions (repeating a vertex), contractions (deleting a repeated vertex) and contiguities (of loops of the same length), in any order. 
It is implicit in the preceding paragraph that we may achieve this result by extending each loop first, then using a contiguity equivalence on their extensions, without shuffling further extensions amongst the steps of the contiguity equivalence. This is justified by the following result, in which we show that operating with extension-contiguity equivalence of loops may be reduced to \emph{trivially} extending the loops and then operating with a contiguity equivalence between the extensions.

\begin{lemma}\label{lem: expand then contract}
    Suppose $l \colon (I_{m}, \{0, m\}) \to (X, x_0)$ and $l' \colon (I_{n}, \{0, n\}) \to (X, x_0)$ are edge loops in $X$ of lengths $m$ and $n$. If $l$ and $l'$ are extension-contiguity equivalent, then we may obtain $l'$ from $l$ by the steps
    $$l \approx l\circ \alpha^{M-m}_m \simeq l'\circ \alpha^{M-n}_n \approx l',$$
    with the first step repeating the final vertex $x_0$ up to some length $M$, the second step a contiguity equivalence of loops of length $M$, and the third step deleting repeats of the final vertex $x_0$ to contract to the length of $l'$.
\end{lemma}

\begin{proof}
    We may re-order the extensions, contractions and contiguities involved in the extension-contiguity, with suitable adjustments, so that all extensions occur first, followed by all contiguities, followed by all contractions.  To see this, suppose we have a contiguity followed by an extension
    $$l \sim l' \approx l'\circ \alpha_i.$$
    Write $l = ( x_0, v_1, \ldots, v_{s-1}, x_0 )$ and $l' = ( x_0, v'_1, \ldots, v'_{s-1}, x_0 )$.  Then the contiguity entails that $\{ v_i, v_{i+1}, v'_i, v'_{i-1} \}$ is a simplex of $X$ for each $i$.  Hence, we have a contiguity 
    $$l\circ \alpha_i = ( x_0, v_1, \ldots, v_i, v_i, \ldots, v_{s-1}, x_0 ) \sim ( x_0, v'_1, \ldots, v'_i, v'_i, \ldots, v'_{s-1}, x_0 ) = l'\circ \alpha_i$$
Part (d) of \propref{prop: contiguity results} gives a contiguity equivalence $l\circ \alpha_0 \simeq l\circ \alpha_i$, so we may obtain the extension contiguity equivalence $l \approx l'\circ \alpha_i$ in the steps 
$$l \approx l\circ \alpha_m \simeq l'\circ \alpha_i.$$
A similar discussion shows that a deletion of a repeated vertex followed by a contiguity may be replaced by a contiguity followed by a deletion of a repeated vertex. In symbols, we may replace the steps $l\circ \alpha_i \approx l \sim l'$ by $l\circ \alpha_i \sim  l'\circ \alpha_i \approx l'$, and apply part (d) of \propref{prop: contiguity results} to use the steps  $l\circ \alpha_i \simeq  l'\circ \alpha_m \approx l'$ instead.
Finally, in moving all extensions to the left and all contractions to the right in the sequence of steps, suppose we encounter an occurrence of a deletion followed by an extension.  Then the steps
$$l\circ \alpha_i \approx l \approx l\circ \alpha_j$$
may be replaced with a contiguity equivalence (between loops of the same length) $l\circ \alpha_i \simeq  l\circ \alpha_j$, once again by part (d) of \propref{prop: contiguity results}.
\end{proof}

From the discussion above, it follows that we may view the edge group as consisting of equivalence classes of suitable simplicial maps of intervals to $X$, under the equivalence relation of extension-contiguity equivalence. The simplicial maps are based loops of all lengths and the contiguities must preserve the endpoints. Furthermore, it is sufficient to use trivial extensions in conjunction with contiguities when operating within an extension-contiguity class of based loops.  This viewpoint on the edge group is transparently equivalent to the (usual) one given earlier. We will free to switch between the two whenever convenient.  

We close this section by giving a third notational device that is convenient for representing loops or paths and contiguities amongst such, and also foreshadows our constructions in the sequel. (Extension-)contiguity equivalence of loops (or of their equivalent simplicial maps) may be represented ``array-style" as follows.

Suppose we have a contiguity equivalence
$$l \sim l^1 \sim \cdots \sim l^{n-1} \sim l'$$
between loops $l$ and $l'$ of length $m$ (recall that contiguity is a relation between loops of the same length only).  Writing each loop in the sequence as a row of $m+1$ vertices of $X$, we obtain the $(m+1)\times(n+1)$ matrix of vertices of $X$
\[
\left[ \begin{array}{c}
l' \\
l^{n-1}  \\
\vdots\\
l^1  \\
l  
\end{array}\right]
\qquad = \qquad
\left[\begin{array}{ccccc}
 x_0  & v'_1 &  \cdots &  v'_{m-1} &  x_0 \\
 x_0  & v^{n-1}_1 &  \cdots &  v^{n-1}_{m-1} &  x_0 \\
 \vdots  & \vdots &  \cdots & \vdots &  \vdots \\
x_0  & v^1_1 &  \cdots &  v^1_{m-1} &  x_0 \\
x_0  & v_1 &  \cdots &  v_{m-1} &  x_0 \\
\end{array}\right].
\]  
The contiguity condition on adjacent loops in the sequence is that, in adjacent rows of this matrix, we have 
$$\{ v^j_i, v^j_{i+1}, v^{j+1}_i, v^{j+1}_{i+1} \}$$
a simplex of $X$, for each $0 \leq i \leq m$ and $0 \leq j \leq n$. 

Now suppose we have extension-contiguity equivalent loops $l$ of length $p$ and $l'$ of length $q$.  This means that we have  trivial extensions $l\circ \alpha_p^{m-p}$ and $l'\circ \alpha_q^{m-q}$ to some common length $m \geq \mathrm{max}\{p, q\}$ and  a contiguity equivalence
$$l\circ \alpha_p^{m-p} \sim l^1 \sim \cdots \sim l^{n-1} \sim l'\circ \alpha_q^{m-q}$$
that may be represented array-style as before. The only difference here is that our matrix now looks like (e.g. if $q < p$)
\[
\left[ \begin{array}{c}
l'\circ \alpha_q^{m-q} \\
l^{n-1}  \\
\vdots\\
l^1  \\
l\circ \alpha_p^{m-p}  
\end{array}\right]
\qquad = \qquad
\left[\begin{array}{ccccccccccc}
 x_0  & v'_1 &  \cdots &  v'_{q-1} &  x_0 & \cdots & x_0 &x_0 & \cdots& x_0& x_0\\
 x_0  & v^{n-1}_1 & \cdots  & v^{n-1}_{q-1}&v^{n-1}_{q} & \cdots & v^{n-1}_{p-1} & v^{n-1}_{p}& \cdots & v^{n-1}_{m-1} &  x_0 \\
 \vdots  & & \vdots & \cdots & \cdots & & \vdots &\vdots &  \vdots \\
 x_0  & v^{1}_1 & \cdots  & v^{1}_{q-1}&v^{1}_{q} & \cdots & v^{1}_{p-1} & v^{1}_{p}& \cdots & v^{1}_{m-1} &  x_0 \\
x_0  & v_1 & \cdots  & v_{q-1}&v_{q} & \cdots & v_{p-1} & x_0& \cdots & x_0 &  x_0 \\
\end{array}\right],
\]  
in which the bottom and top rows, and some of the intermediate rows, have been filled with repeats of $x_0$ to length $m+1$. The contiguity condition here is again that, in adjacent rows of this matrix, we have 
$$\{ v^j_i, v^j_{i+1}, v^{j+1}_i, v^{j+1}_{i+1} \}$$
a simplex of $X$, for each $0 \leq i \leq m$ and $0 \leq j \leq n$. As we work towards the right-hand end of adjacent rows, this condition may reduce---perhaps to a triviality---due to the more frequent repeats of $x_0$ in that part of the matrix.

\section{The simplicial complex $\Omega X$}\label{sec: Omega X}

Given a based simplicial complex $(X, x_0)$, we describe the vertices and simplices of a simplicial complex  $\Omega X$.  Notice that $\Omega X$ will always be an infinite simplicial complex and usually will also be disconnected. 

\textbf{Vertices of $\Omega X$}  are the edge loops in $X$ of length $m$ based at $x_0$, for all $m \geq 0$.   Recall that we write a based  edge loop $l$ of length $m$ as a sequence of vertices
$$l  = (x_0=v_0, v_1, \ldots, v_{m-1}, v_m=x_0),$$
where each $v_i$ is a vertex of $X$ and adjacent pairs $\{ v_i, v_{i+1}\}$ are simplices of $X$ (we allow repeats). By the based loop of length $0$ we mean the singleton $\mathbf{x_0} = ( x_0)$.  The only based loop of length $1$ is $\mathbf{x_0^1} = (x_0, x_0 )$. We write $\mathbf{x^m_0}$ for the edge loop in $X$ of length $m$ that consists of repeats of $x_0$.

\textbf{Simplices of $\Omega X$} are defined as follows: 
Suppose $\sigma = \{ l^0, l^1, \ldots, l^n \}$ is a set of based edge loops in $X$, with at least one of length exactly $m$ and the remainder either of length $m$ or of length $m-1$.  Write each edge loop as
$$l^j = ( x_0 =  v^j_0,  v^j_1, \ldots, v^j_{m-1}, v^j_m = x_0)$$
for each $j = 0, \ldots, n$ (if the length of $l^j$ is $m-1$, then $v^j_{m-1} = x_0$ and there is no $v^j_m$ in this case).  Then $\sigma$ is a simplex of $\Omega X$ if we have
$$\{ v^0_i, v^1_i, \ldots, v^n_i \} \cup \{ v^0_{i+1}, v^1_{i+1}, \ldots, v^n_{i+1} \}$$
a simplex of $X$ for each $i=0, \ldots, m-2$ and   
$$\{ v^0_{m-1}, v^1_{m-1}, \ldots, v^n_{m-1} \} \cup \{ x_0 \}$$
a simplex of $X$.  Then $\sigma$ is an $n$-simplex of $\Omega X$ if all the $l^j$ are distinct as edge loops. 

We will take $\mathbf{x_0}$, the edge loop of length zero, as the basepoint of $\Omega X$.  When we refer to the edge group of $\Omega X$, we intend the edge group of the connected component of $\Omega X$ that contains this basepoint.

\begin{remark}\label{rem: simplices as arrays}
We may represent a simplex of $\Omega X$ in the following way.   As above, suppose we have  $\sigma = \{ l^0, l^1, \ldots, l^n \}$  a set of edge loops in $X$, with at least one of length exactly $m$ and the remainder either of length $m$ or of length $m-1$.  Then write 
$$(l^j)_m = \begin{cases} l^j & \text{if $l^j$ is of length $m$}   \\
& \\
x_0, v^j_1, \ldots, v^j_{m-2}, x_0, x_0  & \text{if $l^j$ is of length $m-1$}. \end{cases}$$
In other words, for the  $l^j$ of length $m-1$, trivially extend each one  to a loop of length $m$ by adding a repeat of $x_0$ at the end.  This has the effect of ``same-sizing" the edge loops of $\sigma$ so they may be displayed as an $(m+1)\times (n+1)$ matrix or grid of vertices of $X$:
\[
\left[ \begin{array}{c}
(l^n)_m  \\
(l^{n-1})_m  \\
\vdots\\
(l^1)_m  \\
(l^{0})_m  
\end{array}\right]
\qquad = \qquad
\left[\begin{array}{ccccc}
 v^n_0 = x_0  & v^n_1 &  \cdots &  v^n_{m-1} &  x_0 \\
 v^{n-1}_0 = x_0  & v^{n-1}_1 &  \cdots &  v^{n-1}_{m-1} &  x_0 \\
 \vdots  & \vdots &  \cdots & \vdots &  \vdots \\
 v^1_0 = x_0  & v^1_1 &  \cdots &  v^1_{m-1} &  x_0 \\
 v^{0}_0 = x_0  & v^{0}_1 &  \cdots &  v^{0}_{m-1} &  x_0 \\
\end{array}\right]
\]  
in which the rows give the edge loops in $X$ (possibly after a trivial extension).  The definition of a simplex of $\Omega X$ above entails that, for $\sigma$ to be a simplex, the union of the vertices from any two adjacent columns must be a simplex of $X$ (not just the vertices from each column individually).  In this way, we may write the matrix of vertices involved column-wise as
$$\left[  \begin{array}{c|c|c|c|c}  \sigma_0 & \sigma_1 & \cdots & \sigma_{m-1} & \sigma_m \end{array} \right]$$   
with the vertices of the typical column 
\[
\sigma_i \qquad = \qquad 
\left[ \begin{array}{c}
v^n_i  \\
v^{n-1}_i  \\
\vdots\\
v^1_i  \\
v^0_i  
\end{array}\right]
\]
forming a simplex $\sigma_i$ of $X$.  The condition for $\sigma$ to be a simplex, then, is that we have
$\sigma_i \cup \sigma_{i+1}$ a simplex of $X$, for $0 \leq i \leq m-1$.  
\end{remark}

\begin{remark}\label{rem: three-simplex}
Edges in $\Omega X$ may join two vertices $l$ and $l'$ that correspond to edge loops of the same length as each other.  In this case, the edge loops $l$ and $l'$ are contiguous as in \defref{def: contiguous edge loops}.  Edges in $\Omega X$ may also join two vertices $l$ and $l'$ that correspond to edge loops whose lengths differ by one from each other. In this case, say $l$ is an edge loop of length $m$ and $l'$ an edge loop of length $m-1$.  Then an edge in $\Omega X$ joins $l$ and $l'$ when we have a contiguity $l \sim l'\circ \alpha_{m-1}$.  Notice that, if we have a contiguity $l \sim l'$ of edge loops of length $m$, then $\{ l, l', l\circ \alpha_m, l'\circ \alpha_m \}$ is a $3$-simplex of $\Omega X$.        
\end{remark}

We generalize the final comment of the above as follows. In the next result and in the sequel, we will write $\overline{l}$ for the trivial extension $l\circ \alpha_m$ where $l$ is a loop of length $m$.  

\begin{lemma}\label{lem: extension of simplex}
Let $\sigma$ be an $n$-simplex of $\Omega X$
$$\sigma = \{ l^0, l^1, \ldots, l^n \}$$
with each $l^j$ a loop of (exactly) length $m$ in $X$. Let 
$$\overline{\sigma} := \{ \overline{l^0}, \overline{l^1}, \ldots, \overline{l^n} \}$$
be the set of vertices of $\Omega X$ given by trivially extending each vertex of $\sigma$. Then
$$\sigma \cup \overline{\sigma} =\{ l^0, l^1, \ldots, l^n, \overline{l^0}, \overline{l^1}, \ldots, \overline{l^n} \}$$
is a $(2n+1)$-simplex of $\Omega X$.  In particular, $\overline{\sigma}$ is an $n$-simplex of $\Omega X$.
\end{lemma}

\begin{proof}
Write out $\sigma \cup \overline{\sigma}$ array-wise as in \remref{rem: simplices as arrays} above, same-sizing the loops involved to be of length $m+1$.  Because we have $\overline{l^{j}} = (l^{0})_{m+1}$ for each $j$, this results in the following double matrix: 
\[
\sigma \cup \overline{\sigma} \qquad = \qquad
\left[ \begin{array}{c}
\overline{l^n}\\
\vdots\\
\overline{l^{0}}\\  
(l^n)_{m+1}  \\
\vdots\\
(l^{0})_{m+1}  
\end{array}\right]
\qquad = \qquad
\left[\begin{array}{cccccc}
 x_0  & v^n_1 &  \cdots &  v^n_{m-1} &  x_0 & x_0 \\
  \vdots  & \vdots &  \cdots & \vdots &  \vdots \\
 x_0  & v^{0}_1 &  \cdots &  v^{0}_{m-1} &  x_0 & x_0\\
  x_0  & v^n_1 &  \cdots &  v^n_{m-1} &  x_0 & x_0 \\
  \vdots  & \vdots &  \cdots & \vdots &  \vdots \\
 x_0  & v^{0}_1 &  \cdots &  v^{0}_{m-1} &  x_0 & x_0\\
\end{array}\right]
\]  
Then the simplex condition as expressed in \remref{rem: simplices as arrays} for $\sigma \cup \overline{\sigma}$ is satisfied because $\sigma$ is a simplex of $\Omega X$, and so satisfies that condition. 
\end{proof}

To operate with the edge group of $\Omega X$, it is sufficient to know the $2$-skeleton.  It is also convenient in some situations to know the $3$-skeleton so that one may check the contiguity condition across disjoint edges of a $3$-simplex. In the next result, we record some specific instances of $2$- and $3$-simplices of $\Omega X$.   

\begin{lemma}\label{lem: 3-simplex extensions}
With the simplices of $\Omega X$ as defined above, we have:
\begin{enumerate}
    \item[(a)] If $l \sim l'$ are two (same-length) contiguous loops in $X$, then $\{l, l', \overline{l}, \overline{l'} \}$ is a simplex of $\Omega X$;
    \item[(b)] If $\{l, l'\}$ is a simplex (an edge) of $\Omega X$ with $l$ of length $m$ and $l'$ of length $m-1$ for some $m$,  then $\{l, l', \overline{l'} \}$ is a simplex of $\Omega X$;
    \item[(c)] If $l \sim l'$ are two contiguous loops of length $m$ in $X$, then $$\{l\circ\alpha_i, l'\circ\alpha_i, l\circ\alpha_{i+1}, l'\circ\alpha_{i+1} \}$$ is a simplex of $\Omega X$, for each $0 \leq i \leq m-1$; 
    \item[(d)] If $l_1 \sim l_1'$ are two contiguous loops (same length $m$) and $\{l_2,  l'_2\}$ is a simplex (an edge) of $\Omega X$ with $l_2$ of length $n$ and $l'_2$ of length either $n$ or $n-1$ for some $n$, then 
    $$\{ l_1\cdot l_2,l_1\cdot l'_2, l'_1\cdot l_2, l'_1\cdot l'_2 \}$$
    is a simplex of $\Omega X$. 
\end{enumerate}
\end{lemma}

\begin{proof}
(a) We remarked on this fact in \remref{rem: three-simplex}.  It is a special case of \lemref{lem: extension of simplex}, since $\{ l, l' \}$ is a simplex (an edge) of $\Omega X$.

(b) This is more-or-less tautological. Suppose that we have 
$$l = x_0, v_1, \cdots, v_{m-1}, x_0 \qquad \mathrm{and} \qquad l' = x_0, v'_1, \cdots, v'_{m-2}, x_0.$$
Write the three loops  after ``same-sizing" them to the longer length $m$ as
\[
\left[ \begin{array}{c}
(l)_{m} = l \\
(l')_{m} = \overline{l'}  \\
(\overline{l'})_{m} = \overline{l'} 
\end{array}\right]
\qquad = \qquad
\left[\begin{array}{cccccc}
x_0  & v_1 &  \cdots & v_{m-2} & v_{m-1} &  x_0 \\
x_0  & v'_1 &  \cdots & v'_{m-2} & x_0 &  x_0 \\
x_0  & v'_1 &  \cdots & v'_{m-2} & x_0 &  x_0 
\end{array}\right].
\]  
The simplex condition we want satisfied is that the union of vertices from adjacent columns should form a simplex of $X$.  Since the second and third rows are repeats, this reduces to the same condition on the first two rows.  But that condition is satisfied for the first two rows, since $\{l, l'\}$ qualified as an edge of $\Omega X$. 

(c) Suppose that we have 
$$l = (x_0, v_1, \cdots, v_{m-1}, x_0) \qquad \mathrm{and} \qquad l' = (x_0, v'_1, \cdots, v'_{m-1}, x_0).$$
Write the four loops  as
\[
\left[ \begin{array}{c}
l\circ\alpha_{i} \\
l'\circ\alpha_{i}  \\
l\circ\alpha_{i+1}  \\
l'\circ\alpha_{i+1}  
\end{array}\right]
\qquad = \qquad
\left[\begin{array}{cccccccc}
x_0  & v_1 &  \cdots &  v_{i} & v_{i} & v_{i+1} & \cdots &  x_0 \\
x_0  & v'_1 &  \cdots &  v'_{i} & v'_{i} & v'_{i+1} & \cdots &  x_0 \\
x_0  & v_1 &  \cdots &  v_{i} & v_{i+1} & v_{i+1} & \cdots &  x_0 \\
x_0  & v'_1 &  \cdots &  v'_{i} & v'_{i+1} & v'_{i+1} & \cdots &  x_0 
\end{array}\right].
\]  
The simplex condition we want satisfied is that the union of vertices from adjacent columns should form a simplex of $X$.  In all but the cases of the $i$th and $(i+1)$st, and the $(i+1)$st and the $(i+2)$nd (indexing with the left-most column as the $0$th), this condition is satisfied as it reduces to the contiguity condition that $l$ and $l'$ satisfy by assumption. For the $i$th and $(i+1)$st columns, as well as for the $(i+1)$st and the $(i+2)$nd columns, the union of vertices is
$$\{ v_i, v'_i, v_{i+1}, v'_{i+1} \},$$
which again is a simplex of $X$ because $l$ and $l'$ satisfy the contiguity condition. The result follows.

(d) Suppose that we have 
$$l_1 = (x_0, v_1, \cdots, v_{m-1}, x_0) \qquad \mathrm{and} \qquad l_1' = (x_0, v'_1, \cdots, v'_{m-1}, x_0),$$
and 
$$l_2 = (x_0, w_1, \cdots, w_{n-1}, x_0) \qquad \mathrm{and} \qquad (l_2')_n = (x_0, w'_1, \cdots, w'_{n-1}, x_0)$$
(if $l'_2$ is of length $n-1$ then we have $w'_{n-1} = x_0$). 
Write these four loops  as
\[
\left[ \begin{array}{c}
l_1\cdot l_2 \\
(l_1\cdot l'_2)_{m+n+1}  \\
l'_1\cdot l_2  \\
(l'_1\cdot l'_2)_{m+n+1}  
\end{array}\right]
\qquad = \qquad
\left[\begin{array}{cccccccccc}
x_0  & v_1 &  \cdots &  v_{m-1} & x_0 & x_0 & w_{1} &  \cdots & w_{n-1} &  x_0 \\
x_0  & v_1 &  \cdots &  v_{m-1} & x_0 & x_0 & w'_{1} &  \cdots & w'_{n-1} &  x_0 \\
x_0  & v'_1 &  \cdots &  v'_{m-1} & x_0 & x_0 & w_{1} &  \cdots & w_{n-1} &  x_0 \\
x_0  & v'_1 &  \cdots &  v'_{m-1} & x_0 & x_0 & w'_{1} &  \cdots & w'_{n-1} &  x_0  
\end{array}\right],
\]  
with $l_1\cdot l'_2$ and $l'_1\cdot l'_2$ same-sized to length $m+n+1$ if need be.
The simplex condition we want satisfied is that the union of vertices from adjacent columns should form a simplex of $X$.  For the $i$th and $(i+1)$st columns for $i=0, \ldots, m-1$, and $i=m+1, \ldots, m+n$ respectively, this condition is satisfied as it reduces to the contiguity condition $l_1 \sim l_1'$, $l_2 \sim l_2'$ respectively. For the $m$th and $(m+1)$st columns, all entries are $x_0$ and the condition is trivially satisfied.  The result follows.
\end{proof}

\begin{remark}
In general we expect $\Omega X$ to be disconnected.  Based on topological intuition, this corresponds to components of $\Omega X$ being identified with the fundamental group of $X$. (We confirm this intuition in \corref{cor: components and edge group} below.) Furthermore, any component of 
$\Omega X$ will have vertices that correspond to loops of infinitely many lengths because, for any vertex $l$, we have an edge $\{ l, \overline{l} \}$ in $\Omega X$.    
\end{remark}

\begin{example}\label{ex: hollow 3-simplex}
Suppose $X$ is a hollow $3$-simplex with vertices $\{ x_0, v_1, v_2, v_3 \}$ (so any triple of vertices forms a simplex of $X$ but all $4$ vertices do not).  Let $l^1$, $l^2$ and $l^3$ be the edge loops written, as in \remref{rem: simplices as arrays} above, in the form
\[
\left[ \begin{array}{c}
l^3  \\
l^{2}  \\
l^{1}  
\end{array}\right]
\qquad = \qquad
\left[\begin{array}{cccc}
 x_0  & v_1 &  v_3 &  x_0 \\
 x_0  & v_1 &  v_1 &  x_0 \\
 x_0  & v_1 &  v_2 &  x_0 \\
\end{array}\right]
\]  
Then $\{ l^1, l^2 \}$, $\{ l^1, l^3 \}$ and $\{ l^2, l^3 \}$ are edges in $\Omega X$, so that $\{ l^1, l^2, l^3 \}$ is a $3$-clique.  However, $\{ l^1, l^2, l^3 \}$ is not a simplex of $\Omega X$.  To see this write---as in \remref{rem: simplices as arrays}---the matrix of vertices column-wise as 
\[
\left[ \begin{array}{c}
l^3  \\
l^{2}  \\
l^{1}  
\end{array}\right]
\qquad = \qquad
\left[\begin{array}{c|c|c|c}
 \sigma_0  & \sigma_1 &  \sigma_2 &  \sigma_3 
\end{array}\right]
\]  
 and note that $\sigma_2 \cup \sigma_3 = \{v_3, v_2, v_1 \} \cup \{ x_0 \}$ is not a simplex of $X$.
\end{example}

 \begin{remark}
We may think of $\Omega X$ as being ``stratified" by the sub-complexes whose vertices are loops of length up to $m$.  If we write $\Omega X[m]$ for the sub-complex whose vertices are loops of length exactly $m$, then the first type of edge (contiguity) describes the edges of this sub-complex.  The second type of edge (trivial extension and  contiguity) gives connections between these strata or sub-complexes.  
\end{remark} 

\begin{proposition}\label{prop:Omega-functor}
The construction $\Omega(-)$ defines an endofunctor on the category of pointed simplicial sets and simplicial maps.
\end{proposition}

\begin{proof}
% Let $\sigma   
% 
% Now we may consider each $l^{i}$ to be a simplicial map, $(I_{m}, \{0,1\}) \rightarrow (X, x_{0})$.
% Our condition that $\sigma$ be a simplex now becomes:  

Suppose that $f : (X, x_{0}) \rightarrow (Y, y_{0})$ is a simplicial map.
A mapping $\Omega f : \Omega X \rightarrow \Omega Y$ is defined on vertices by composition:  $\Omega f (l) = f \circ l$, where $l : (I_{m}, \{0,1\}) \rightarrow (X,x_{0})$.
By definition, if $g : (Y, y_{0}) \rightarrow (Z, z_{0})$ is another simplicial map, then 
$\Omega (g \circ f) = \Omega g \circ \Omega f$, and clearly $\Omega (1_{X}) = 1_{\Omega X}$.

We need to show that $\Omega f$ is simplicial.  
Let $\sigma= \{ l^{0}, \ldots, l^{k} \}$ be a $k$-simplex in $\Omega X$, in which one vertex has length $m$ and the rest
have length either $m$ or $m-1$.
With the notation of Remark~\ref{rem: simplices as arrays}, 
the fact that trivial extensions commute with simplicial maps means that for all $i$, $(f \circ l^{i})_{m} = f \circ (l^{i})_{m}$.
Our condition that $\sigma$ is a simplex is that for each simplex $\epsilon$ in $I_{m}$, $(l^{0})_{m}(\epsilon) \cup \cdots \cup (l^{k})_{m}(\epsilon)$ is a simplex in $X$.

Now, $\Omega f (\sigma) = \{ \Omega f (l^{0}), \ldots, \Omega f (l^{k}) \}
= \{ f \circ l^{0}, \ldots, f \circ l^{k} \}$.
If $\epsilon$ is a simplex in $I_{m}$, then 
\[
	(f \circ l^{0})_{m}(\epsilon) \cup \cdots \cup 
    (f \circ l^{k})_{m}(\epsilon) = f((l^{0})_{m}(\epsilon) \cup \cdots \cup (l^{k})_{m}(\epsilon)).
\]
Since $\sigma$ is a simplex, $(l^{0})_{m}(\epsilon) \cup \cdots \cup (l^{k})_{m}(\epsilon)$ is a simplex in $X$.
Since $f$ is simplicial, $f((l^{0})_{m}(\epsilon) \cup \cdots \cup (l^{k})_{m}(\epsilon))$ is a simplex in $Y$.
Therefore, $\Omega f (\sigma)$ is a simplex in $\Omega Y$, and so $\Omega f$ is simplicial.
\end{proof}

\section{Edge Loops in $\Omega X$ and components of $\Omega X$}\label{sec: loops in Omega X}

Recall that we write $\mathbf{x^m_0} \in \Omega X[m]$ for the edge loop in $X$ of length $m$ that consists of repeats of $x_0$.  
A typical edge loop (or path) in $\Omega X$ will be denoted by $\gamma$, and (if of length $n$) is of the form
$$\gamma = (\mathbf{x_0}, \mathbf{x^1_0}, l^2, \ldots, l^{n-2}, \mathbf{x^1_0}, \mathbf{x_0}),$$
with each $l^j$ an edge loop in $X$ (of varying lengths).

\begin{proposition}\label{prop: loop in X(M)}
Suppose we have an edge loop in $\Omega X$ of length $n$
$$\gamma = ( \mathbf{x_0}, \mathbf{x^1_0}, l^2, \ldots,  l^{n-2}, \mathbf{x^1_0}, \mathbf{x_0} ).$$
Let $M$ be the least $m$ such that $\gamma \subseteq \Omega X[\leq m]$.  (Said differently, the longest edge loop in $\gamma$ is of length $M$.) Notice that this entails $n \geq 2M$.  Then there is some edge loop $\widehat{\gamma}$ with  $\widehat{\gamma} \approx \gamma$ and $\widehat{\gamma}$ of the form
$$\widehat{\gamma}     = ( \mathbf{x_0}, \mathbf{x^1_0}, \ldots, \mathbf{x^M_0}, \ell^1, \ldots, \ell^p, \mathbf{x^M_0}, \ldots, \mathbf{x^1_0}, \mathbf{x_0} ),$$
with $( \mathbf{x^M_0}, \ell^1, \ldots, \ell^p, \mathbf{x^M_0} )$ an edge loop in $\Omega X [M]$ starting and finishing at  $\mathbf{x^M_0} \in \Omega X[ M]$.
\end{proposition}

\begin{proof}
First we argue that we may assume the lengths of the edge loops in $\gamma$ are non-decreasing up to the maximum length of $M$, and then non-increasing back down to $0$ once their lengths start to decrease.
To this end, consider the (integer) sequence of lengths of each vertex $l^i$ of $\gamma$:
$$0, 1, \mathrm{length}(l^2), \cdots, M, \cdots, \mathrm{length}(l^{n-2}), 1, 0.$$ 
Suppose that, somewhere between $l^2$ and $l^{n-2}$ we have a section of $\gamma$
$$\ldots, l^{j}, \ldots, l^{j+k}, \ldots $$
whose lengths display a local minimum, in the sense that we have
$$\mathrm{length}(l^{j}) = r, \mathrm{length}(l^{j+1}) = \cdots = \mathrm{length}(l^{j+k-1}) = r-1, \mathrm{length}(l^{j+k}) = r  $$
for some $r$.  Then we may replace this section of $\gamma$ with the path
$$\ldots, l^{j}, \overline{l^{j}}, \ldots, \overline{l^{j+k-1}}, l^{j+k}, \ldots,$$
each vertex of which has length $r$.  This adjusted version of $\gamma$ is contiguous, as a loop in $\Omega X$, to the original $\gamma$ since the pair agree outside the section we are adjusting, and the original section and its replacement satisfy the contiguity condition for paths by parts (a) and (b) of \lemref{lem: 3-simplex extensions}.  By repeating this removal of any local minima in the sequence of edge lengths, we may assume, up to extension-contiguity equivalence of loops in $\Omega X$,  that  $\gamma$ is a sequence of edge loops whose lengths are non-decreasing up to the maximum length of $M$, continue at length $M$ for a section, then continue non-increasing back down to $0$.  

Next, assume $\gamma$ is now a sequence of edge loops in $X$ whose lengths are non-decreasing up to the maximum length of $M$. Working up to extension-contiguity of loops, we may remove any repeats of  $\mathbf{x_0}$ and $\mathbf{x^1_0}$ and write this first section of $\gamma$ as
$$\mathbf{x_0}, \mathbf{x^1_0}, l^2_1, \ldots, l^2_{n_2}, l^3_1, \ldots, l^3_{n_3}, \ldots, l^{M-1}_1, \ldots, l^{M-1}_{n_{M-1}}, l^{M}_1, \ldots, l^{M}_{n_{M}}, \ldots,$$
after which the lengths of the edge loops are non-increasing back down to $0$. Here, we intend an edge loop $l^{i}_{j}$ to be an edge loop of length $i$, of which there are $n_i$ in $\gamma$. We now work on this section of $\gamma$ using the two types of move we have available for operating with the edge group.  Begin by repeating $\mathbf{x^1_0}$ and then trivially extend every loop from the $2$nd occurrence of $\mathbf{x^1_0}$ through $l^{M-1}_{n_{M-1}}$ (the last occurrence of a loop of length $M-1$ before the length $M$ loops start). At this point, we have replaced $\gamma$ with a loop in $\Omega X$ that starts  
$$\mathbf{x_0}, \mathbf{x^1_0}, \mathbf{x^2_0},  \overline{l^2_1}, \ldots,  \overline{l^{M-1}_{n_{M-1}}},  l^{M}_{n_{M}}, \ldots,$$
with the edge loops from $l^{M}_{n_{M}}$ onwards those of  $\gamma$. After the repeating of $\mathbf{x^1_0}$, trivially extending all the terms we did results in a contiguous loop in $\Omega X$ from parts (a) and (b) of \lemref{lem: 3-simplex extensions}. 

We iterate this step, repeating $\mathbf{x^2_0}$ and then trivially extending every loop from the $2$nd occurrence of $\mathbf{x^2_0}$ through $\overline{l^{M-2}_{n_{M-2}}}$ (which is now the last occurrence of a loop of length $M-1$ before the length $M$ loops start). As before, this results in a loop that is extension-contiguity equivalent to the original $\gamma$.  We iterate this step sufficiently many times until we arrive at a loop that is extension-contiguity equivalent to the original $\gamma$, and which starts with a section
$$\mathbf{x_0}, \mathbf{x^1_0}, \ldots, \mathbf{x^M_0},$$
and continues with sections of loops in $X$ all of length $M$
$$l^i_1 \circ \alpha_i^{M-i}, \ldots, l^i_{n_i}\circ \alpha_i^{M-i},$$
for $i = 2, \ldots, M-1$, followed by the section 
$$l^{M}_1, \ldots, l^{M}_{n_{M}}$$
from the original $\gamma$, followed by the remainder of the original $\gamma$.   
We operate in a similar way on the remainder of the original $\gamma$ to arrive at the desired result.  Each step in this process consists of an extension followed by a contiguity, so we arrive at a loop $\widehat{\gamma}$ in $\Omega X$ of the desired form that is extension-contiguity equivalent to the original $\gamma$.
\end{proof}

We now discuss (edge-path) connected components of $\Omega X$. The next result is a combinatorial counterpart of the familiar adjunction in the topological setting
$$\mathrm{map}\left(I, \mathrm{map}(I, X) \right) \equiv \mathrm{map}(I \times I, X)$$
that allows a homotopy of paths in $X$ to be viewed as a path in $\Omega X$. 

\begin{proposition}\label{prop: left-right homotopy}
Let  
$$\gamma = ( l^0, l^1, \ldots,  l^{N-1}, l^N )$$
be an edge path (not necessarily an edge loop) in $\Omega X$ of length $N$.  Then $l^0$ and $l^N$ are extension-contiguity equivalent as loops in $X$.  Conversely, if $l$ and $l'$ are extension-continguity equivalent loops in $X$, then there is an edge path (of some length $N$) in $\Omega X$
$$( l, l^1, \ldots,  l^{N-1}, l' ).$$
\end{proposition}

\begin{proof}
Let $M$ be the maximum length of loop amongst the vertices $l^j$ of $\gamma$.  For each loop $l^j$ of length $m_j$, extend it to a loop of length $M$ (if necessary) by setting
$$(l^j)_M = \begin{cases} l^j & \text{if $m_j = M$}   \\
& \\
l^{j} \circ \alpha_{m_{j}}^{M-m_{j}}  & \text{if $m_j < M$}. \end{cases}$$
Consecutive vertices $l^j$ and $l^{j+1}$ of the path $\gamma$ form an edge in $\Omega X$ and so satisfy the (simplex) condition that 
$$\{ v^j_i, v^j_{i+1}, v^{j+1}_i, v^{j+1}_{i+1} \}$$
is a simplex of $X$, for each $0\leq i \leq m_j$ and $1 \leq j \leq N-1$. Recall that adjacent vertices $l^j$ and $l^{j+1}$ in $\Omega X$ may only differ in length by at most $1$. Since we are only adding repeats of $x_0$ to the ends of the loops $l^j$ and $l^{j+1}$, it follows that each pair $\{ (l^j)_M, (l^{j+1})_M \}$ satisfies the same condition (now with $i \leq M-1$).  But this is the same condition that must be satisfied for $(l^j)_M$ and $(l^{j+1})_M$ to be contiguous loops in $X$ (now of the same length as each other). That is, we have a contiguity equivalence
$$(l^0)_M \sim \cdots \sim (l^{N-1})_M \sim (l^N)_M.$$
Since (tautologically) we have $l^0 \approx (l^0)_M$ and $l^N \approx (l^N)_M$, it follows that we have $l^0 \approx l^N$. 

Conversely, suppose that $l$ of length $p$ and $l'$ of length $q$ are extension-contiguity equivalent loops in $X$. Then there are extensions $(l)_m$ and $(l')_m$ of $l$ and $l'$, respectively, to some common length $m \geq \mathrm{max}\{ p, q \}$ that are contiguity equivalent. Now, tautologically, we have extension-contiguity equivalences
$$l \approx (l)_{p+1} \approx \cdots \approx (l)_m \qquad \mathrm{and} \qquad l' \approx (l')_{q+1} \approx \cdots \approx (l')_m.$$
It follows that we have an edge path $\gamma$ in $\Omega X$ from vertex $l$ to vertex $l'$
\begin{displaymath}
    \gamma = ( l, (l)_{p+1}, \ldots, (l)_{m}, \ldots, (l')_{m}, \cdots, (l')_{q+1}, l' ). \qedhere
\end{displaymath}
\end{proof}

We may restate \propref{prop: left-right homotopy} as follows. 

\begin{corollary}\label{cor: components and edge group}
For each simplicial complex $X$, there is a bijection of sets
\begin{displaymath}
\left\{ \mbox{Edge-Path Components of $\Omega X$} \right\} \leftrightarrow E(X) \qed
\end{displaymath}
\end{corollary}

We record some items related to the ideas of \propref{prop: left-right homotopy} that we will use in the sequel.

\begin{lemma}\label{lem: contiguity of extensions}
Let  
$$\gamma = ( \mathbf{x_0^r}, l^1, \ldots,  l^{N-1}, \mathbf{x_0^s} )$$
be an edge path (not necessarily an edge loop) in $\Omega X$ of length $N$.  
Let $M$ be the longest length of loop in $X$ that occurs as a vertex of $\gamma$, and  write 
$$[\gamma]_M := ( \mathbf{x_0^M}, (l^1)_M, \ldots,  (l^{N-1})_M, \mathbf{x_0^M} ).$$
\begin{itemize}
    \item[(a)] $[\gamma]_M$ is a loop in $\Omega X[M]$ based at $\mathbf{x_0^M}$, the constant loop in $X$ of length $M$.
 \item[(b)] Suppose that 
$$\gamma' = ( \mathbf{x_0^p}, k^1, \ldots,  k^{N-1}, \mathbf{x_0^q} )$$
is a second path with $\gamma' \sim \gamma$ (contiguous paths in $\Omega X$ of the same length). Notice this entails $r$ and $p$, respectively $s$ and $q$, differ by no more than $1$. Let $M$ be at least the longest length of loop in $X$ that occurs as a vertex of either $\gamma$ or $\gamma'$. Then we have
$[\gamma]_M \sim [\gamma']_M$ (contiguous as loops in $\Omega X[M]$ based at $\mathbf{x_0^M}$).
\end{itemize}
\end{lemma}

\begin{proof}
(a) This is just a particular case of \propref{prop: left-right homotopy}.

(b) The contiguity $\gamma \sim \gamma'$ entails that $\{ l^j, l^{j+1}, k^j, k^{j+1} \}$ is a simplex of $\Omega X$, for each $0 \leq j \leq N-1$ (we interpret $l^0$ as $\mathbf{x_0^r}$ and so-on). As a simplex, the lengths of these four loops may differ from each other by no more than $1$.  Suppose the vertex that has longest length amongst the four has length $m_j$.  We same-size all four and write them array-style as above, as 
\[
\left[ \begin{array}{c}
(l^j)_{m_j} \\
{(l^{j+1})}_{m_j} \\
{(k^{j})}_{m_j} \\
{(k^{j+1})}_{m_j}  
\end{array}\right]
\qquad = \qquad
\left[\begin{array}{ccccc}
x_0  & v^j_1 &  \cdots &  v^j_{m_j - 1} &  x_0 \\
x_0  & v^{j+1}_1 &  \cdots &  v^{j+1}_{m_j - 1} &   x_0 \\
x_0  & w^j_1 &  \cdots &  w^j_{m_j - 1} &   x_0 \\
x_0  & w^{j+1}_1 &  \cdots &  w^{j+1}_{m_j - 1} &   x_0 
\end{array}\right],
\]  
where one or more of the penultimate entries of each row may be $x_0$, if that row corresponds to a loop of length $m_j -1$.  Then the simplex condition amounts to the union of those vertices of $X$ in adjacent columns giving a simplex of $X$.  Now, if we extend each row to length $M$ by adding repeats of $x_0$, the same simplex condition will hold.   
\end{proof}

To show that different components of $\Omega X$ have isomorphic edge group,  we will use \emph{left translation by a vertex of $\Omega X$}.  Let $\ell \in \Omega X$ be a loop in $X$ based at $x_0$.  By left translation by $\ell$ we mean the vertex map $L_\ell \colon \Omega X \to \Omega X$ defined by $L_\ell(l) := \ell\cdot l$ (the concatenation of loops in $X$) for each vertex $l \in \Omega X$.   

\begin{lemma}\label{lem: left translation}
For $\ell$ any vertex of $\Omega X$, left translation by $\ell$ extends to a simplicial map $L_\ell \colon \Omega X \to \Omega X$.   If $\ell \sim \ell'$ are contiguous loops in $X$, so that $\{ \ell, \ell' \}$ is an edge of $\Omega X$, then we have contiguous simplicial maps $L_\ell \sim L_{\ell'} \colon \Omega X \to \Omega X$.   If $(\Omega X)_l$ denotes the edge-path component of $\Omega X$ that contains the vertex $l \in \Omega X$, then $L_\ell \left( (\Omega X)_l \right) \subseteq (\Omega X)_{L_\ell(l)}$. 
\end{lemma} 

\begin{proof}
Suppose that 
$$\sigma = \{ l^0, \cdots, l^n \}$$
is a simplex of $\Omega X$. For the first assertion, we want to confirm that $L_\ell(\sigma)$ is a simplex of $\Omega X$.  For the second assertion, we want to confirm that $L_\ell(\sigma) \cup L_{\ell'}(\sigma)$ is a simplex of $\Omega X$.  This last condition reduces to the first by taking $\ell' = \ell$, so we will just establish the more general statement.     Suppose that  at least one of the $l^j$ is a loop in $X$ of length exactly $m$ and the remainder are loops in $X$ either of length $m$ or of length $m-1$. Then the matrix
\[
\sigma_m = 
\left[ \begin{array}{c}
(l^n)_m  \\
(l^{n-1})_m  \\
\vdots\\
(l^1)_m  \\
(l^{0})_m  
\end{array}\right]
\qquad = \qquad
\left[\begin{array}{ccccc}
 x_0  & v^n_1 &  \cdots &  v^n_{m-1} &  x_0 \\
 x_0  & v^{n-1}_1 &  \cdots &  v^{n-1}_{m-1} &  x_0 \\
 \vdots  & \vdots &  \cdots & \vdots &  \vdots \\
 x_0  & v^1_1 &  \cdots &  v^1_{m-1} &  x_0 \\
 x_0  & v^{0}_1 &  \cdots &  v^{0}_{m-1} &  x_0 \\
\end{array}\right]
\]  
satisfies the simplex condition:  the union of the vertices of $X$ from any two adjacent columns is a simplex of $X$.  Then $L_\ell(\sigma)\cup L_{\ell'}(\sigma)$ may be represented array style as
\[
\left[ \begin{array}{c} L_\ell(\sigma) \\
L_{\ell'}(\sigma) \end{array}\right]
= 
\left[ \begin{array}{c}
\ell\cdot(l^n)_m  \\
\ell\cdot(l^{n-1})_m  \\
\vdots\\
\ell\cdot(l^1)_m  \\
\ell\cdot(l^{0})_m \\ 
\ell'\cdot(l^n)_m  \\
\ell'\cdot(l^{n-1})_m  \\
\vdots\\
\ell'\cdot(l^1)_m  \\
\ell'\cdot(l^{0})_m 
\end{array}\right]
\]  
and we check this also satisfies the simplex condition. To this end, suppose that $\ell = (x_0, v_1, \ldots, v_{r-1}, x_0)$  and $\ell' = (x_0, v'_1, \ldots, v'_{r-1}, x_0)$ both have length $r$ (contiguity of loops entails they are of the same length as each other). Then the union of vertices from columns $i$ and $i+1$ of this matrix, for $0 \leq i \leq r-1$ consists only of
$$\{ v_i, v_{i+1}, v'_i, v'_{i+1} \},$$
which is a simplex of $X$ from the contiguity $\ell \sim \ell'$. The union of vertices from columns $i$ and $i+1$ for $r+1 \leq i \leq m+r+1$ consists of the union of vertices from two adjacent columns of $\sigma_m$---again a simplex of $X$ since $\sigma$ is a simplex.  The union of vertices from columns $r$ and $r+1$  consists of $\{x_0 \}$, since $\ell$, $\ell'$ and each $(l^j)_m$ is a loop in $X$ based at $x_0$. Finally, observe that we have 
$$\ell\cdot (l^j)_m = (\ell\cdot  l^j)_{r+m+1} \qquad \text{and} \qquad \ell'\cdot (l^j)_m = (\ell'\cdot  l^j)_{r+m+1},$$
since the effect of ``same-sizing" the rows in the matrix is simply to add an $x_0$ at the right-hand end where needed, in either case.  It follows that 
$$L_\ell(\sigma) \cup L_{\ell'}(\sigma) = \{ \ell\cdot l^0, \ldots, \ell\cdot l^n, \ell'\cdot l^0, \ldots, \ell'\cdot l^n  \}$$
is a simplex of $\Omega X$.

The final assertion, that $L_\ell$ preserves edge-path components, is true of any simplicial map and not just these translations in $\Omega X$. Any simplicial map $f \colon X \to Y$ takes an edge path in $X$ to an edge path in $Y$. It follows that any vertex $v'$ of $X$ that is in the edge-path component of a vertex $v$ will be mapped to $f(v')$ in the edge-path component of $Y$ that contains $f(v)$.
\end{proof}

Notice that the composition of left translations is again a left translation, since we have
$$L_{\ell'} \circ L_\ell = L_{\ell' \cdot \ell} \colon \Omega X \to \Omega X,$$
for vertices $\ell, \ell' \in \Omega X$ where $\ell' \cdot \ell$ denotes concatenation of based loops in $X$.   

\begin{remark}[On Translation in Loop Spaces]
In the topological (continuous) setting, $\Omega X$ is an $H$-space whose multiplication $\mu$ (derived from composition of loops) restricts to left- and right-translation maps $L_\ell$ and $R_\ell$, respectively, by a typical loop $\ell \in \Omega X$:
$$\xymatrix{\Omega X \ar[d]_{\text{incl.}} \ar[rd]^{L_\ell} \\
\{\ell\} \times \Omega X \ar[r]_(0.6){\mu} & X}
\qquad \text{and} \qquad \xymatrix{\Omega X \ar[d]_{\text{incl.}} \ar[rd]^{R_\ell} \\
\Omega X \times \{\ell\}  \ar[r]_(0.6){\mu} & X}
$$
In our simplicial setting, whereas \lemref{lem: left translation} provides some encouragement, in fact right translation by a typical element of $\Omega X$ fails to give a simplicial map.  For instance, suppose we have $X$ a cycle graph with $4$ vertices $\{ x_0, v_1, v_2, v_3 \}$ and edges $\{ (x_0, v_1), (v_1, v_2), (v_2, v_3), (v_3, x_0) \}$.  Take $\ell \in \Omega X$ to be the loop $\ell = (x_0, v_1, v_2, v_3, x_0)$.  In $\Omega X$, we have a simplex $\sigma = \{ \mathbf{x_0}, \mathbf{x_0^1} \}$ (an edge in $\Omega X$).  But the vertex map $R_\ell \colon \Omega X \to \Omega X$, defined by $R_\ell(l) := l\cdot \ell$ for a vertex $l \in \Omega X$ (namely, a loop in $X$), gives 
$$R_\ell(\sigma) = \{   (x_0, x_0, v_1, v_2, v_3, x_0), (x_0, x_0, x_0, v_1, v_2, v_3, x_0) \}.$$
This is not a simplex (an edge) of $\Omega X$, since $v_2$ is not adjacent to $x_0$ in $X$ (and neither is $v_1$ adjacent to $v_3$ in $X$). Similar examples illustrate that the vertex map $\Omega X \times \Omega X \to \Omega X$ given by concatenation of based loops in $X$ does not extend to a simplicial map.  However, as we show next, \emph{right translation by a trivial loop} in $\Omega X$ does give a simplicial map.
\end{remark}

Let $(\Omega X)_0$ denote the edge-path component of the basepoint $\mathbf{x_0} \in \Omega X$. 

\begin{lemma}\label{lem: left vs right translation}
For any $N \geq 0$, let $R_\mathbf{x_0^N} \colon \Omega X \to \Omega X$ be the vertex map defined by $R_\mathbf{x_0^N}(l) := l\cdot \mathbf{x_0^N}$ (right translation by the trivial loop of length $N$).  Then $R_\mathbf{x_0^N}$ extends to a simplicial map of $\Omega X$.  If we restrict $R_\mathbf{x_0^N}$ to $\Omega X[M]$, the subcomplex of $\Omega X$ whose vertices are based loops in $X$ of length exactly $M$,  then we have a contiguity equivalence of simplicial maps $L_\mathbf{x_0^N} \simeq R_\mathbf{x_0^N} \colon \Omega X[M] \to \Omega X[M+N+1]$.  Furthermore, both $L_\mathbf{x_0^N}$ and $R_\mathbf{x_0^N}$ map from the edge-path component $(\Omega X)_0$ of $\Omega X$ to itself and induce the same homomorphism of edge groups, namely, we have
$$(L_\mathbf{x_0^N})_* = (R_\mathbf{x_0^N})_* \colon E\left( (\Omega X)_0, \mathbf{x_0} \right) \to  E\left( (\Omega X)_0, \mathbf{x_0^{N+1}} \right).$$
\end{lemma}

\begin{proof}
Suppose that 
$\sigma = \{ l^0, \cdots, l^n \}$
is a simplex of $\Omega X$. For the first assertion, we want to confirm that $R_\mathbf{x_0^N}(\sigma)$ is a simplex of $\Omega X$.     Suppose that  at least one of the $l^j$ is a loop in $X$ of length exactly $m$ and the remainder are loops in $X$ either of length $m$ or of length $m-1$. Then  the matrix
\[
\sigma_m = 
\left[ \begin{array}{c}
(l^n)_m  \\
(l^{n-1})_m  \\
\vdots\\
(l^1)_m  \\
(l^{0})_m  
\end{array}\right]
\qquad = \qquad
\left[\begin{array}{ccccc}
 x_0  & v^n_1 &  \cdots &  v^n_{m-1} &  x_0 \\
 x_0  & v^{n-1}_1 &  \cdots &  v^{n-1}_{m-1} &  x_0 \\
 \vdots  & \vdots &  \cdots & \vdots &  \vdots \\
 x_0  & v^1_1 &  \cdots &  v^1_{m-1} &  x_0 \\
 x_0  & v^{0}_1 &  \cdots &  v^{0}_{m-1} &  x_0 \\
\end{array}\right]
\]  
satisfies the simplex condition:  the union of the vertices of $X$ from any two adjacent columns is a simplex of $X$.  Now $R_\mathbf{x_0^N}$ maps the vertices of $\sigma$ to vertices of $\Omega X$ that are loops in $X$ of length $m+N+1$ and possibly of length $m+N$.  If $l^j$ has length $m$, then we have $(l^j)_m = l^j$ and $\left(R_\mathbf{x_0^N}(l^j)\right)_{m+N+1} = R_\mathbf{x_0^N}(l^j) = R_\mathbf{x_0^N}\left( (l^j)_m\right)$.   If $l^j$ has length $m-1$, then we have $(l^j)_m = l^j\cdot \mathbf{x_0}$ and $\left(R_\mathbf{x_0^N}(l^j)\right)_{m+N+1} = l^j \cdot \mathbf{x_0^N} \cdot \mathbf{x_0} = l^j \cdot \mathbf{x_0}\cdot \mathbf{x_0^N}  = R_\mathbf{x_0^N}\left( (l^j)_m\right)$. So, $R_\mathbf{x_0^N}(\sigma)$ may be represented array style as
$$\left[\begin{array}{cccccccc}
 x_0  & v^n_1 &  \cdots &  v^n_{m-1} &  x_0 & x_0 & \cdots & x_0 \\
 x_0  & v^{n-1}_1 &  \cdots &  v^{n-1}_{m-1} &  x_0  & x_0 &\cdots & x_0 \\
 \vdots  & \vdots &  \cdots & \vdots &  \vdots & \vdots &  \cdots & \vdots\\
 x_0  & v^1_1 &  \cdots &  v^1_{m-1} &  x_0   & x_0 & \cdots & x_0\\
 x_0  & v^{0}_1 &  \cdots &  v^{0}_{m-1} &  x_0 & x_0 & \cdots & x_0 \\
\end{array}\right].$$
This evidently satisfies the simplex condition, given that the left half of the matrix does.  

Now we show the contiguity equivalence of simplicial maps $L_\mathbf{x_0^N} \simeq R_\mathbf{x_0^N} \colon \Omega X[M] \to \Omega X[M+N+1]$. 
Refer to the notation from \defref{def: extensions}.  Restrict to vertices of $\Omega X$ that are in $\Omega X[M]$, and define vertex maps
$$\alpha_i^* \colon \Omega X[M] \to \Omega X[M+1]$$
for each $i = 0, \ldots, M$ by setting $\alpha_i^*(l) := l\circ \alpha_i$.   Namely, if $l = (x_0, v_1, \ldots, v_i, \ldots, v_{M-1}, x_0)$, then $\alpha_i^*(l) = (x_0, v_1, \ldots, v_i, v_i, \ldots, v_{M-1}, x_0)$ (repeat vertex $v_i$). 
Notice that we have $\alpha_0^* = L_\mathbf{x_0}$ and $\alpha_M^* = R_\mathbf{x_0}$. First we show that we  have a contiguity of simplicial maps $\alpha_i^* \sim \alpha_{i+1}^* \colon \Omega X[M] \to \Omega X[M+1]$ for each $i = 0, \ldots, M_1$.
Suppose that 
$\sigma = \{ l^0, \cdots, l^n \}$
is a simplex of $\Omega X[M]$. We want to confirm that $\alpha_i^* (\sigma) \cup \alpha_{i+1}^* (\sigma)$ is a simplex of $\Omega X$ (actually of $\Omega X[M+1]$).     Each $l^j$ is a loop in $X$ of length exactly $M$. Then  the matrix
\[
\sigma = 
\left[\begin{array}{ccccc}
 x_0  & v^n_1 &  \cdots &  v^n_{M-1} &  x_0 \\
 x_0  & v^{n-1}_1 &  \cdots &  v^{n-1}_{M-1} &  x_0 \\
 \vdots  & \vdots &  \cdots & \vdots &  \dots \\
 x_0  & v^1_1 &  \cdots &  v^1_{M-1} &  x_0 \\
 x_0  & v^{0}_1 &  \cdots &  v^{0}_{M-1} &  x_0 \\
\end{array}\right] = 
\left[\begin{array}{ccccc}
 \mathbf{v_0}  & \mathbf{v_1} &  \cdots &  \mathbf{v_{M-1}} &  \mathbf{v_0} 
\end{array}\right]
\]  
satisfies the simplex condition:  the union of the vertices of $X$ from any two adjacent columns is a simplex of $X$.  In the above, we have used column vector notation for the columns of the matrix $\sigma$.  With the same notation, the union $\alpha_i^* (\sigma) \cup \alpha_{i+1}^* (\sigma)$ may be represented array style as the ``double" matrix
\[
\left[\begin{array}{c}
 \alpha_i^* (\sigma)\\
 \alpha_{i+1}^* (\sigma)
\end{array}\right] = 
\left[\begin{array}{ccccccccc}
\mathbf{v_0}  & \mathbf{v_1} &  \cdots &  \mathbf{v_i} & \mathbf{v_i} & \mathbf{v_{i+1}} & \cdots & \mathbf{v_{M-1}} &  \mathbf{v_0} \\
 \mathbf{v_0}  & \mathbf{v_1} &  \cdots &  \mathbf{v_i} & \mathbf{v_{i+1}} & \mathbf{v_{i+1}} & \cdots & \mathbf{v_{M-1}} &  \mathbf{v_0}
\end{array}\right].
\]  
It is easy to see that vertices from adjacent columns of this matrix have union a simplex of $X$, just as the same condition is satisfied by the columns of $\sigma$. 
The contiguity  $\alpha_i^* \sim \alpha_{i+1}^* \colon \Omega X[M] \to \Omega X[M+1]$ follows.  Thus we have a contiguity equivalence
$$L_\mathbf{x_0} = \alpha_0^* \sim \alpha_{1}^* \sim \cdots \sim \alpha_{M}^* =  R_\mathbf{x_0} \colon \Omega X[M] \to \Omega X[M+1].$$
This contiguity equivalence $L_\mathbf{x_0} \simeq   R_\mathbf{x_0}$ may be extended to one 
$$L_\mathbf{x^N_0} \simeq   R_\mathbf{x^N_0} \colon \Omega X[M] \to \Omega X[M+N+1]$$
by writing 
$$L_\mathbf{x^N_0} =  L_\mathbf{x_0} \circ L_\mathbf{x_0} \circ \cdots \circ L_\mathbf{x_0} \colon \Omega X[M] \to \Omega X[M+1]\to \cdots \to \Omega X[M+N+1]$$
and likewise for $R_\mathbf{x^N_0}$, and applying  part (b) of \propref{prop: contiguity results}.

For any $M \geq 0$, write $\gamma^M_0$ for the edge path in $(\Omega X)_0$
$$\gamma^M_0 = (\mathbf{x_0}, \mathbf{x^1_0}, \ldots, \mathbf{x^{M}_0}).$$
Both $L_{\mathbf{x_0^{N}}}$ and $R_{\mathbf{x_0^{N}}}$ map the component $(\Omega X)_0$ of $\Omega X$ to itself, since we have $L_{\mathbf{x_0^{N}}} (\mathbf{x_0}) = R_{\mathbf{x_0^{N}}}(\mathbf{x_0}) = \mathbf{x_0^{N+1}}$, which is connected to $\mathbf{x_0}$ by the path $\gamma^{N+1}_0$.  Hence we have induced homomorphisms of edge groups
$$(L_\mathbf{x_0^N})_* , (R_\mathbf{x_0^N})_* \colon E\left( (\Omega X)_0, \mathbf{x_0} \right) \to  E\left( (\Omega X)_0, \mathbf{x_0^{N+1}} \right).$$
We will eventually show that these are the same \emph{isomorphism}, but for the time being will just show they are the same homomorphism.
 
By \propref{prop: loop in X(M)}, we may assume that a typical  $[\alpha] \in E\left( (\Omega X)_0, \mathbf{x_0}\right)$ is represented by a loop that is the concatenation of edge paths in $\Omega X$
$$\alpha = \gamma^M_0\cdot \gamma \cdot \widetilde{\gamma^M_0},$$
for some $M$ and with middle section
$$\gamma = (\mathbf{x^{M}_0}, l^1, \ldots, l^p, \mathbf{x^{M}_0})$$
an edge loop in $\Omega X[M]$ based at $\mathbf{x^{M}_0}$. Then 
 $(L_{\mathbf{x_0^{N}}})_* ([\alpha])$ may be represented by the concatenation of edge paths in $(\Omega X)_0$
$$L_{\mathbf{x_0^{N}}} (\gamma^M_0) \cdot L_{\mathbf{x_0^{N}}} (\gamma) \cdot L_{\mathbf{x_0^{N}}} ( \widetilde{\gamma^M_0}).$$
Now $L_{\mathbf{x_0^{N}}} (\gamma^M_0) = R_{\mathbf{x_0^{N}}} (\gamma^M_0)$ and  $L_{\mathbf{x_0^{N}}} ( \widetilde{\gamma^M_0}) = R_{\mathbf{x_0^{N}}} ( \widetilde{\gamma^M_0})$, as all vertices involved here consist of repeats of $x_0$.   For the middle section, we may view $\gamma$ as a simplicial map $\gamma \colon I_{p+1} \to \Omega X[M]$ and apply part (b) of \propref{prop: contiguity results} together with the contiguity equivalence shown above to write a contiguity equivalence
$$L_{\mathbf{x_0^{N}}} (\gamma) = L_{\mathbf{x_0^{N}}} \circ \gamma \simeq R_{\mathbf{x_0^{N}}} \circ \gamma = R_{\mathbf{x_0^{N}}} (\gamma)$$
of loops in $\Omega X[M]$.  Furthermore, this contiguity equivalence of loops  may be seen to leave the endpoints fixed at $\mathbf{x_0^{M+N+1}}$.  Thus, it may be spliced into a contiguity equivalence 
$$
\begin{aligned}L_{\mathbf{x_0^{N}}} (\alpha)  &= L_{\mathbf{x_0^{N}}} (\gamma^M_0) \cdot L_{\mathbf{x_0^{N}}} (\gamma) \cdot L_{\mathbf{x_0^{N}}} ( \widetilde{\gamma^M_0})\\
&= R_{\mathbf{x_0^{N}}} (\gamma^M_0) \cdot L_{\mathbf{x_0^{N}}} (\gamma) \cdot R_{\mathbf{x_0^{N}}} ( \widetilde{\gamma^M_0})\\
&\simeq R_{\mathbf{x_0^{N}}} (\gamma^M_0) \cdot R_{\mathbf{x_0^{N}}} (\gamma) \cdot R_{\mathbf{x_0^{N}}} ( \widetilde{\gamma^M_0}) = R_{\mathbf{x_0^{N}}} (\alpha).
\end{aligned}
$$
The equality $(L_\mathbf{x_0^N})_* = (R_\mathbf{x_0^N})_* $ of homomorphisms follows. 
\end{proof}

 Let $\ell \in \Omega X$ be any vertex. As a simplicial map, $L_\ell$ maps the edge-path component $(\Omega X)_0$ to the edge-path component $(\Omega X)_{\ell\cdot x_0} = (\Omega X)_{\overline{\ell}}$.  Now this is the same edge-path component as $(\Omega X)_{\ell}$, since $\{ \ell, \overline{\ell} \}$ is a simplex (an edge) of $\Omega X$.  Thus, we have an induced homomorphism of edge groups
$$(L_\ell)_* \colon E\left( (\Omega X)_0, \mathbf{x_0}\right) \to  E\left( (\Omega X)_\ell, \overline{\ell} \right).$$

\begin{theorem}\label{thm: homogeneous components}
Let $(\Omega X)_0$ denote the edge-path component of the constant loop $\mathbf{x_0}$ and let $(\Omega X)_\gamma$ denote the edge-path component that contains  $\ell$, some edge loop in $X$, as a vertex. Then we have an isomorphism of edge groups
$$\xymatrix@1{E\left((\Omega X)_0, \mathbf{x_0}\right) \ar[r]^{(L_\ell)_*}_{\cong} &  E\left( (\Omega X)_\ell, \overline{\ell} \right) \ar[r]^{\Phi_e}_{\cong} & E\left((\Omega X)_\ell, \ell\right),}$$
with $\Phi_e$ the change of basis isomorphism from $e = ( \ell,  \overline{\ell} )$, the edge path of length $1$ from $\ell$ to $\overline{\ell}$ in $\Omega X$.  
\end{theorem}

\begin{proof}
Let $\widetilde{\ell}$ denote the reverse loop of $\ell$ and suppose each is a loop of length $r$ in $X$.  As we observed above, we have $L_{\widetilde{\ell}} \circ L_\ell = L_{\widetilde{\ell}\cdot \ell}$.  Now we have $L_{\widetilde{\ell}\cdot \ell}(\mathbf{x_0}) = 
\widetilde{\ell}\cdot \ell\cdot \mathbf{x_0} =  \overline{\widetilde{\ell}\cdot\ell}$.  But this vertex of $\Omega X$ is in the same edge-path component of $\Omega X$ as $\mathbf{x_0}$.  To see this, recall the standard 
contiguity equivalence $\mathbf{x^{2r+1}_0} \sim L_1 \sim \cdots \sim L_{r-1} \sim \widetilde{\ell}\cdot\ell$ of loops in $X$ 
used to show that the reverse loop gives the inverse in the edge group, from \lemref{lem: reverse as inverse}.   Then trivially extending each loop in this contiguity equivalence gives an edge path in $\Omega X$
$$\eta = \left( \mathbf{x^{2r+2}_0}, L_1\cdot\mathbf{x_0}, \ldots, L_{r-1}\cdot\mathbf{x_0}, \overline{\widetilde{\ell}\cdot\ell} \right)$$ 
from $\mathbf{x_0^{2r+1}}$ to $\widetilde{\ell}\cdot \ell$. Then we may concatenate this with the edge path $(\mathbf{x_0}, \mathbf{x^1_0}, \ldots, \mathbf{x^{2r+2}_0})$ to display an edge path in $\Omega X$ from $\mathbf{x_0}$ to $\overline{\widetilde{\ell}\cdot\ell}$.
It follows from this discussion that we have a homomorphism
$$(L_{\widetilde{\ell}})_* \circ (L_\ell)_* = (L_{\widetilde{\ell}\cdot \ell})_* \colon  E\left( (\Omega X)_0, \mathbf{x_0}\right) \to E\left( (\Omega X)_0, \overline{\widetilde{\ell}\cdot\ell}\right).$$
First we will argue that this homomorphism is an isomorphism.  

The contiguity equivalence $\widetilde{\ell}\cdot\ell \sim  L_{r-1} \sim \cdots \sim L_1 \sim  \mathbf{x^{2r+1}_0}$ of loops in $X$ implies a corresponding contiguity equivalence of simplical maps
$$L_{ \widetilde{\ell}\cdot\ell} \sim L_{L_{r-1}} \sim \cdots \sim L_{L_1} \sim  L_{\mathbf{x^{2r+1}_0}} \colon \Omega X \to \Omega X,$$
from \lemref{lem: left translation}.  Furthermore, the edge path $\eta$ in $\Omega X$ displayed above is exactly the ``trace" of the basepoint $\mathbf{x_0}$ of $\Omega X$ under the maps in this contiguity equivalence (in reverse order).  From \lemref{lem: contiguous induced hom}, therefore, we obtain the following commutative diagram of homomorphisms of edge groups: 
\begin{equation}\label{eq: translation iso1}
\xymatrix{
E\left( (\Omega X)_0, \mathbf{x_0}\right) \ar[r]^{(L_{\widetilde{\ell}\cdot \ell})_*} \ar[rd]_{(L_{\mathbf{x^{2r+1}_0}})_*}  & E\left( (\Omega X)_0, \overline{\widetilde{\ell}\cdot\ell}\right) \ar[d]^{\Phi_\eta}_{\cong} \\
  &    E\left( (\Omega X)_0, \mathbf{x^{2r+2}_0}\right)
}
\end{equation}

From \lemref{lem: left vs right translation}, we may replace the homomorphism $(L_{\mathbf{x^{2r+1}_0}})_*$ in (\ref{eq: translation iso1}) by $(R_{\mathbf{x^{2r+1}_0}})_*$, its right-translation counterpart.
Now we show that the following diagram of homomorphisms of edge groups commutes:
\begin{equation}\label{eq: translation iso2}
\xymatrix{
E\left( (\Omega X)_0, \mathbf{x_0}\right) \ar[r]^{(R_{\mathbf{x^{2r+1}_0}})_*} \ar@{=}[rd]_{\mathrm{id}}  & E\left( (\Omega X)_0, \mathbf{x^{2r+2}_0}\right)
 \ar[d]^{\Phi_{\gamma^{2r+2}_0}}_{\cong} \\
  &    E\left( (\Omega X)_0, \mathbf{x_0}\right)
}
\end{equation}
It will follow that all the homomorphisms displayed in (\ref{eq: translation iso1}) and (\ref{eq: translation iso2}) are isomorphisms.

As in \lemref{lem: left vs right translation}, for any $M \geq 0$, write $\gamma^M_0$ for the edge path in $(\Omega X)_0$
$$\gamma^M_0 = (\mathbf{x_0}, \mathbf{x^1_0}, \ldots, \mathbf{x^{M}_0}).$$
By \propref{prop: loop in X(M)}, we may assume that a typical  $[\alpha] \in E\left( (\Omega X)_0, \mathbf{x_0}\right)$ is represented by a loop that is the concatenation of edge paths in $\Omega X$
$$\alpha = \gamma^M_0\cdot \gamma \cdot \widetilde{\gamma^M_0},$$
for some $M$ and with middle section
$$\gamma = (\mathbf{x^{M}_0}, l^1, \ldots, l^p, \mathbf{x^{M}_0})$$
an edge loop in $\Omega X[M]$ based at $\mathbf{x^{M}_0}$. Then 
 $\Phi_{\gamma_0^{2r+2}} \circ (R_{\mathbf{x_0^{2r+1}}})_* ([\alpha])$ may be represented by the concatenation of edge paths in $(\Omega X)_0$
$$\gamma^{2r+2}_0 \cdot  R_{\mathbf{x_0^{2r+1}}} (\gamma^M_0) \cdot R_{\mathbf{x_0^{2r+1}}} (\gamma) \cdot R_{\mathbf{x_0^{2r+1}}} ( \widetilde{\gamma^M_0}) \cdot \widetilde{\gamma^{2r+2}_0}$$
that, by removing a repeat of the vertex $\mathbf{x_0^{2r+2}}$ of $\Omega X$ from the start of the section $R_{\mathbf{x_0^{2r+1}}} (\gamma^M_0)$ and the end of its reverse, is extension-contiguity equivalent to
$$\gamma^{2r+2+M}_0 \cdot   R_{\mathbf{x_0^{2r+1}}} (\gamma) \cdot \widetilde{\gamma^{2r+2+M}_0}.$$
We  now show by induction that there is an extension-contiguity equivalence
$$\gamma^{2r+2+M}_0 \cdot R_{\mathbf{x_0^{2r+1}}} (\gamma)\cdot  \widetilde{\gamma^{2r+2+M}_0} \approx \gamma^{2r+2+M-i}_0 \cdot R_{\mathbf{x_0^{2r+1-i}}} (\gamma) \cdot  \widetilde{\gamma^{2r+2+M-i}_0}$$
for each $i = 0, \ldots, 2r+2$.  By $R_{\mathbf{x_0^{-1}}}$ we intend the identity. Induction starts with $i=0$, where there is nothing to show. Now consider the section 
$$
\begin{aligned}
( \mathbf{x^{2r+2+M-(i+1)}_0}, &\  \mathbf{x^{2r+2+M-i}_0}, \\
\mathbf{x^{M}_0} \cdot \mathbf{x^{2r+1-i}_0},& \ l_1\cdot \mathbf{x^{2r+1-i}_0}, \ldots, l_p\cdot \mathbf{x^{2r+1-i}_0}, \mathbf{x^{M}_0} \cdot \mathbf{x^{2r+1-i}_0},\\
\ \  \mathbf{x^{2r+2+M-i}_0},& \  \mathbf{x^{2r+2+M-(i+1)}_0})
\end{aligned}
$$
of $\gamma^{2r+2+M-i}_0 \cdot R_{\mathbf{x_0^{2r+1-i}}} (\gamma)  \cdot  \widetilde{\gamma^{2n+2+M-i}_0}$ from the penultimate vertex of $\gamma^{2r+2+M-i}_0$ through the $2$nd vertex of $\widetilde{\gamma^{2r+2+M-i}_0}$.

Then part (a) of \lemref{lem: 3-simplex extensions} gives a contiguity between this section and the path 
$$
\begin{aligned}
& (\mathbf{x^{2r+2+M-(i+1)}_0}, \  \mathbf{x^{2r+2+M-(i+1)}_0}, \\
& \ \ \mathbf{x^{M}_0} \cdot \mathbf{x^{2r+1-(i+1)}_0}, \ l_1\cdot \mathbf{x^{2r+1-(i+1)}_0}, \ldots, l_p\cdot \mathbf{x^{2r+1-(i+1)}_0}, \mathbf{x^{M}_0} \cdot \mathbf{x^{2r+1-(i+1)}_0},\\
&\ \ \ \  \mathbf{x^{2r+2+M-(i+1)}_0}, \  \mathbf{x^{2r+2+M-(i+1)}_0} )
\end{aligned}
$$
Since this contiguity leaves the endpoints fixed, it may be spliced into a contiguity
$$
\begin{aligned}
&\gamma^{2n+2+M-i}_0 \cdot R_{\mathbf{x_0^{2r+1-i}}} (\gamma) \cdot  \widetilde{\gamma^{2r+1+M-i}_0} \\
& \sim \gamma^{2r+2+M-(i+1)}_0 \cdot \mathbf{x^{2r+2+M-(i+1)}_0} \cdot R_{\mathbf{x_0^{2r+1-(i+1)}}} (\gamma) \cdot \mathbf{x^{2r+2+M-(i+1)}_0}\cdot  \widetilde{\gamma^{2r+2+M-(i+1})_0}.
\end{aligned}
$$
Working up to extension-contiguity equivalence, we may remove the repeats of $\mathbf{x^{2r+2+M-(i+1)}_0}$ to obtain an extension-contiguity equivalence
$$
\begin{aligned}
\gamma^{2r+2+M-i}_0 \cdot &R_{\mathbf{x_0^{2r+1-i}}} (\gamma) \cdot  \widetilde{\gamma^{2r+2+M-i}_0} \\
 &\approx \gamma^{2r+2+M-(i+1)}_0  \cdot R_{\mathbf{x_0^{2r+1-(i+1)}}} (\gamma) \cdot  \widetilde{\gamma^{2r+2+M-(i+1})_0}.
\end{aligned}
$$
This completes the induction.  It follows that we have the extension-contiguity equivalence
$$  
\gamma^{2r+2+M}_0 \cdot R_{\mathbf{x_0^{2r+1}}} (\gamma) \cdot  \widetilde{\gamma^{2r+2+M}_0} \approx \gamma^{M}_0 \cdot R_{\mathbf{x_0^{-1}}} (\gamma) \cdot  \widetilde{\gamma^{M}_0} = \alpha
$$
that shows  the diagram (\ref{eq: translation iso2}) commutes.  Combining diagrams (\ref{eq: translation iso1})  and (\ref{eq: translation iso2}), it follows that 
$(L_{\mathbf{x_0^{2r+1}}})_*$, $(R_{\mathbf{x_0^{2r+1}}})_*$ and 
$(L_{\widetilde{\ell}})_* \circ (L_\ell)_* = (L_{\widetilde{\ell}\cdot \ell})_*$ are all isomorphisms of edge groups and hence that 
$(L_\ell)_* \colon E\left((\Omega X)_0, \mathbf{x_0}\right) \to  E\left( (\Omega X)_\ell, \overline{\ell} \right)$ is injective.

A similar argument, \emph{mutatis mutandis}, shows that we also have a 
commutative diagram of homomorphisms as follows.
\begin{equation}\label{eq L_gamma iso2}
\xymatrix{
E\left( (\Omega X)_\ell, \ell\right) \ar[r]^{(L_{\widetilde{\ell}})_*} \ar@{=}[rrd]_{\mathrm{id}} & E\left( (\Omega X)_0, \widetilde{\ell}\cdot \ell\right) \ar[r]^{(L_{\ell})_*} & E\left( (\Omega X)_\ell, \ell\cdot \widetilde{\ell}\cdot\ell \right) \ar[d]^{\Phi_\zeta}_{\cong} \\
  &   & E\left( (\Omega X)_\ell, \ell\right)
}
\end{equation}
Here, the change of basepoint isomorphism $\Phi_\zeta$ is that induced by the path $\zeta = (\ell)\cdot L_\ell(\gamma_0^{2r+1}) \cdot L_\ell(\Gamma)$, 
with $\Gamma = (\mathbf{x^{2r+1}_0}, L_1, \ldots, L_{n-1},\widetilde{\ell}\cdot\ell)$ the standard path used in \lemref{lem: reverse as inverse}. (The contiguity equivalence of \lemref{lem: reverse as inverse} is translated into a path in $\Omega X$ as in \propref{prop: left-right homotopy}.) We omit the details of this as they involve the same ingredients as  above.

Now (\ref{eq L_gamma iso2}) implies that 
$$(L_{\ell})_* \colon E\left( (\Omega X)_0, \widetilde{\ell}\cdot \ell\right) \to  E\left( (\Omega X)_\ell, \ell\cdot \widetilde{\ell}\cdot\ell \right)$$
is surjective.  But the following diagram commutes, as is easily checked:
$$\xymatrix{
E\left( (\Omega X)_0, \widetilde{\ell}\cdot \ell\right) \ar[r]^{(L_{\ell})_*}  &  E\left( (\Omega X)_\ell, \ell\cdot \widetilde{\ell}\cdot\ell \right)  \\
E\left( (\Omega X)_0, \mathbf{x_0}\right) \ar[u]_{\cong}^{\Phi_{\widetilde{\lambda}}} \ar[r]_{(L_\ell)_*}  &  E\left( (\Omega X)_\ell, \overline{\ell} \right) \ar[u]_{\cong}^{\Phi_{L_{\widetilde \gamma}(\lambda)}}
}
$$
The vertical maps are change of basepoint isomorphisms induced by the (reverses of) the path $\lambda = \gamma_0^{2r+1} \cdot \Gamma$  and its translate  $L_{\ell}(\lambda)$. It follows that both horizontal maps are surjections.  Thus, we deduce that
$$(L_\gamma)_* \colon E\left( (\Omega X)_0, \mathbf{x_0}\right)  \to  E\left( (\Omega X)_\gamma, \overline{\gamma}\right)$$
is an isomorphism which, with a further change of basepoint isomorphism 
$$E\left( (\Omega X)_\gamma, \overline{\gamma}\right) \cong E\left( (\Omega X)_\gamma, \gamma\right)$$
gives the desired isomorphism of edge groups.
\end{proof}

\section{Face Spheres in a Simplicial Complex; the Face Group}\label{sec: face group}

In related work, the first-named author and others have described a counterpart to the edge group of a simplicial complex that corresponds to the second homotopy group as the edge group corresponds to the fundamental group. In \cite{L-S-S}, we describe a group $F(X, x_0)$ associated to a (based) simplicial complex $X$---the \emph{face group} of $X$---that satisfies $F(X, x_0) \cong \pi_2\left(|\Omega X|, x_0\right)$.
For details about this face group see \cite{L-S-S}.  We give a brief description of it here. 

Consider simplicial maps of the form
$$f \colon \left( I_{m} \times I_{n}, \partial (I_{m} \times I_{n}) \right) \to (X, x_0),$$
for various $m, n$. Here, $I_{m} \times I_{n}$ is the categorical product of intervals considered as simplicial complexes as before: $I_m$ consists of the integers $\{0, \ldots, m\}$ as vertices and pairs of consecutive integers are the edges. So, for example, $I_1 \times I_1$ is the $3$-simplex $\{ (0,0), (1, 0), (0, 1), (1,1) \}$.  Also, $\partial (I_{m} \times I_{n})$ denotes the boundary of the rectangle $I_{m} \times I_{n}$ in the sense that
$$\partial (I_{m} \times I_{n}) = 
\{ 0\} \times I_n \cup \{ m\} \times I_n \cup I_m \times \{0\} \cup I_m \times \{n\}.$$
Then the maps (of pairs of simplicial complexes) that we consider restrict to the constant map at $x_0$ on the subcomplex $\partial (I_{m} \times I_{n})$ of $I_{m} \times I_{n}$, so we have
$$f \left( \partial (I_{m} \times I_{n}) \right) = \{x_0\}.$$
Contiguity equivalence gives an equivalence relation on all such maps  defined on the same-sized rectangle.  Namely, for  
$$f, g \colon \left( I_{m} \times I_{n}, \partial (I_{m} \times I_{n}) \right) \to (X, x_0),$$
we have $f \simeq g$ if there are maps
$$f, f_1, \ldots, f_n, g \colon \left( I_{m} \times I_{n}, \partial (I_{m} \times I_{n}) \right) \to (X, x_0)$$
and a sequence of contiguities
$f \sim f_{1} \sim \cdots \sim f_{n} \sim g$.

Now suppose we have a simplicial map $f \colon \left(I_m \times I_n, \partial (I_m \times I_n) \right) \to (X,x_0)$. For any $r, s \geq 0$, we may view  $I_m \times I_n \subseteq I_{m+r} \times I_{n+s}$ as a sub-complex. We say (simplicial) 
$\bar f\colon \left(I_{m+r} \times I_{n+s}, \partial (I_{m+s} \times I_{n+s}) \right) \to (X,x_0)$ is a \emph{trivial extension of  $f$} when the vertex map of $\bar{f}$ is given by
\[ \bar f(x) = \begin{cases}
f(x) & \text{ if $x\in I_m \times I_n \subseteq I_{m+r} \times I_{n+s}$,} \\
x_0 & \text{ otherwise.}
\end{cases} \]

Given simplicial maps $f\colon  \left(I_m \times I_n, \partial (I_m \times I_n) \right) \to (X,x_0)$ and $g\colon \left(I_{m'} \times I_{n'}, \partial (I_{m'} \times I_{n'}) \right) \to (X,x_0)$, we  say that $f$ and $g$ are \emph{extension-contiguity equivalent}, and write $f \approx g$, when there exist $\bar m \ge \max(m,m')$ and $\bar n \ge \max(n,n')$ and $\bar f, \bar g: I_{\bar m, \bar n} \to X$ with $\bar f$ a trivial extension of $f$ and $\bar g$ a trivial extension of $g$ and $\bar f$ is contiguity equivalent to $\bar g$ by a  contiguity equivalence relative to the boundary.

In Theorem 2.4 of \cite{L-S-S} we show that extension-contiguity equivalence of maps is an equivalence relation on the set of maps $\left(I_m \times I_n, \partial (I_m \times I_n) \right) \to (X,x_0)$ (all shapes and sizes of rectangle).
Then we write $F(X,x_0)$ for the set of equivalence classes of simplicial maps $\left(I_{m} \times I_{n}, \partial (I_{m} \times I_{n}) \right)\to (X,x_0)$, for all  $I_{m} \times I_{n}$, modulo the equivalence relation of extension-contiguity equivalence.

A binary operation in $F(X,x_0)$ is induced by the following operation on maps. Let $f\colon \left(I_{m} \times I_{n}, \partial (I_{m} \times I_{n}) \right) \to (X,x_0)$ and $g\colon \left(I_{r} \times I_{s}, \partial (I_{r} \times I_{s}) \right) \to (X,x_0)$ be simplicial maps.  Define $f\cdot g\colon \left(I_{m+r+1} \times I_{n+s+1}, \partial (I_{m+r+1} \times I_{n+s+1}) \right)\to (X, x_0)$ on vertices by
$$
(f\cdot g)(i,j)=
    \begin{cases}
        f(i,j) & \text{if } (i,j)\in [0,m]_{\Z}\times [0,n]_{\Z}\\
        g(i-(m+1), j-(n+1)) & \text{if } (i,j)\in [m+1,m+r+1]_{\Z}\times [n+1,n+s+1]_{\Z}\\
        x_0 & \text{otherwise}
    \end{cases}
$$
and extend as a simplicial map over each simplex of $I_{m+r+1} \times I_{n+s+1}$.
In Theorem 3.3 of \cite{L-S-S} we show that $F(X,x_0)$ with this operation is a group, the \emph{face group of} $(X, x_0)$.

There is a transparent correspondence between face spheres of size $m \times n$ and loops of length $n$ in $\Omega X[m]$, given by matching the rows of the face sphere with the vertices of the loop. Suppose we have  a face sphere $f \colon \left( I_{m} \times I_{n}, \partial (I_{m} \times I_{n}) \right) \to (X, x_0)$. Then we may define $\gamma_f \colon I_n \to \Omega X[m] \subseteq \Omega X$, a loop of length $n$ in $\Omega X[m]$ based at the vertex $\mathbf{x_0^m}$ of $\Omega X[m]$, by setting
$$\gamma_f := (\mathbf{x_0^m}=l^0, l^1, \ldots, l^{n-1}, l^n = \mathbf{x_0^m}),$$
where $l^j(s) = f(s, j)$ for $j = 0, \ldots, n$. In the other direction, given a loop $\gamma\colon I_n \to \Omega X[m] \subseteq \Omega X$ based at the vertex $\mathbf{x_0^m}$ of $\Omega X[m]$, we define a face sphere 
$f_\gamma \colon \left( I_{m} \times I_{n}, \partial (I_{m} \times I_{n}) \right) \to (X, x_0)$ by setting
$$f_\gamma(i, j) = \gamma(j)(i)$$
for $0 \leq i \leq m$ and $0 \leq j \leq n$.

\begin{lemma}\label{lem: contiguity loops face spheres}
With the notation above, we have
$\gamma \sim \gamma'$
as based loops in $\Omega X[m]$ if, and only if, we have
$$f_\gamma \sim f_{\gamma'}\colon I_m \times I_n \to X$$
as simplicial maps.
\end{lemma}

\begin{proof}
Suppose that 
$$\gamma = ( \mathbf{x_0^m}, l^1, \ldots,  l^{n-1}, \mathbf{x_0^m} ) \quad \text{and} \quad \gamma' = (\mathbf{x_0^m}, k^1, \ldots,  k^{n-1}, \mathbf{x_0^m} ),$$
with each $l^j$ and each $k^j$ a loop of length $m$ in $X$.  For $\gamma \sim \gamma'$, the contiguity condition is that 
$$\{ l^j, l^{j+1}, k^j, k^{j+1} \}$$
is a simplex of $\Omega X$, for each $j = 0, \ldots, n-1$. Writing these four loops array-style and interpreting the simplex condition gives that 
$$\{ l^j(i), l^j(i+1), l^{j+1}(i), l^{j+1}(i+1), k^j(i), k^j(i+1), k^{j+1}(i), k^{j+1}(i+1) \}$$
is a simplex of $X$, for each $i = 0, \ldots, m-1$. and $j = 0, \ldots, n-1$.  This is the condition for $f_\gamma$ and $f_{\gamma'}$ to be contiguous as maps $I_m \times I_n \to X$.
\end{proof}

We will want some basic ways of operating with face spheres.   Recall the notation of \defref{def: extensions}.
Now suppose we have  a face sphere $f \colon \left( I_{m} \times I_{n}, \partial (I_{m} \times I_{n}) \right) \to (X, x_0)$. Then the composition
$$f \circ (\alpha^r_i \times \alpha^s_j) \colon \left( I_{m+r} \times I_{n+s}, \partial (I_{m+r} \times I_{n+s}) \right) \to (X, x_0)$$
is the face sphere of size $(m+r)\times (n+s)$ obtained from $f$---when viewed ``array-style" as an $(m+1)\times (n+1)$ array of values in $X$---by repeating the $i$th column of values $r$ times and the $j$th row of values $s$ times.
In particular, the compositions $f \circ (\alpha^r_m \times \alpha^s_n)$ give typical trivial extensions of $f$. 

In the following result, we are mainly interested in the case in which the map is a face sphere, but there is no need to restrict to that case for the result. 

\begin{lemma}
 Let  $f \colon I_m \times I_n \to X$ be a simplicial map. For given $r, s \geq 0$, we have a contiguity equivalence
 $$f \circ (\alpha^r_i \times \alpha^s_j) \simeq f \circ (\alpha^r_k \times \alpha^s_l)\colon I_{m+r} \times I_{n+s} \to X$$
 for each $0 \leq i, k \leq m$ and each $0 \leq j, l \leq n$. 
\end{lemma}

\begin{proof}
The assertion is that repeating different columns the same number of times and/or repeating different rows the same number of times leads to contiguity equivalent maps. Since we may obtain a composition $f \circ (\alpha^r_i \times \alpha^s_j)$ by successively repeating column $i$ and successively repeating row $j$ (in either order), it is sufficient to show that we have contiguities 
 $$f \circ (\alpha^1_i \times \mathrm{id}) \sim f \circ (\alpha^1_{i+1} \times \mathrm{id})\colon I_{m+1} \times I_{n} \to X$$
and
 $$f \circ ( \mathrm{id}\times\alpha^1_j ) \sim f \circ ( \mathrm{id}\times \alpha^1_{j+1} )\colon I_{m} \times I_{n+1} \to X$$
for each $0 \leq i \leq m$ and each $0\leq j\leq n$.

For the first of these, which simply asserts that repeating the $i$th column or repeating the $(i+1)$st column gives contiguous maps, we check the contiguity condition directly. Let $\sigma = \{ (s, t), (s+1, t), (s, t+1), (s+1, t+1) \}$ be the typical $3$-simplex of $I_{m+1} \times I_{n}$. Unless $s = i$ or $s=i+1$, both $f \circ (\alpha^1_i \times \mathrm{id})$ and $f \circ (\alpha^1_{i+1} \times \mathrm{id})$ agree on each vertex of $\sigma$: the contiguity condition is trivially satisfied. If $s=i$, then we have
$$
\begin{aligned}
    f \circ (\alpha^1_i \times \mathrm{id})(\sigma) \cup f \circ (\alpha^1_{i+1} \times \mathrm{id})(\sigma) &= 
\{ f(s, t), f(s, t+1), f(s+1, t), f(s+1, t+1)\}\\
&\ \ \cup \{ f(s, t), f(s, t+1), f(s+1, t), f(s+1, t+1)\}\\
&= \{ f(s, t), f(s, t+1), f(s+1, t), f(s+1, t+1)\}\\
&= f(\sigma),
\end{aligned}$$
which is a simplex of $X$.  Likewise, we check that 
$$
    f \circ (\alpha^1_i \times \mathrm{id})(\sigma) \cup f \circ (\alpha^1_{i+1} \times \mathrm{id})(\sigma) = 
f(\sigma).
$$
The contiguity $f \circ (\alpha^1_i \times \mathrm{id}) \sim f \circ (\alpha^1_{i+1} \times \mathrm{id})$ follows.  The contiguity of maps with repeated rows follows from the same argument, since we may transpose the arrays of values that represent the maps and write:
$$\begin{aligned}
 \left(f \circ ( \mathrm{id}\times\alpha^1_j )\right)^T &=  f^T \circ (\alpha^1_j \times \mathrm{id})  \\
 &\sim f^T \circ (\alpha^1_{j+1} \times \mathrm{id})\\
 &=\left(f \circ ( \mathrm{id}\times\alpha^1_{j+1} )\right)^T.
\end{aligned}$$
Now the simplices of $I_m \times I_{n+1}$ and $I_{n+1} \times I_m$ correspond to each other under this transposition, and it is clear that, in this context, maps will be contiguous exactly when their transposes are contiguous.  
\end{proof}

\section{Identifying Face Group of $X$ with Edge Group of $\Omega X$}\label{sec: F(X) iso E(Omega X)}

The correspondence between face spheres in $X$ and edge loops in $\Omega X$ observed above \lemref{lem: contiguity loops face spheres} strongly suggests the familiar adjunction in the topological setting
$$\mathrm{map}\left(I, \mathrm{map}(I, X) \right) \equiv \mathrm{map}(I \times I, X),$$
which, when suitably restricted, leads to the isomorphism $\pi_2(X) \cong \pi_1(\Omega X)$.  We will establish this isomorphism in our combinatorial setting.

In \lemref{lem: contiguity loops face spheres}, we passed from a face sphere in $X$ to a sequence of edge loops in $\Omega X[m]$.   If we are to pass from face spheres in $X$ to edge loops in $\Omega X$, we need to connect this sequence of vertices in $\Omega X[m]$ to the basepoint proper $\mathbf{x_0} \in \Omega X$, in the style of \propref{prop: loop in X(M)}. So, define a function
$$\Phi \colon \left\{ \text{face spheres in $X$ of size } m \times n \right\} \to  \left\{ \text{edge loops in $\Omega X$ of length } 2m + n \right\}$$
by setting (in the notation from above \lemref{lem: contiguity loops face spheres})   
$$\Phi(f)     = ( \mathbf{x_0}, \mathbf{x^1_0}, \ldots, \mathbf{x^m_0}, l^1, \ldots, l^{n-1}, \mathbf{x^m_0}, \ldots, \mathbf{x^1_0}, \mathbf{x_0} ),$$
where $\gamma_f     = ( \mathbf{x^m_0}, l^1, \ldots, l^{n-1}, \mathbf{x^m_0} )$. In the notation of \secref{sec: loops in Omega X}, we may write $\Phi(\gamma)$ as the concatenation of paths in $\Omega X$
$$\Phi(\gamma) = \gamma_0^{m} \cdot \gamma \cdot \widetilde{\gamma_0^{m}}.$$

Note that the edge loop $\Phi(f)$ in $\Omega X$ is based at $\mathbf{x_0}$ and is an edge loop in the component $(\Omega X)_0$.   

In the following result, we assume $X$ is a connected simplicial complex and the edge group $E(\Omega X, \mathbf{x_0})$ is the edge group of the connected component of $\mathbf{x_0}$.  Per \thmref{thm: homogeneous components}, this is isomorphic to the edge group of any other connected component of $\Omega X$.  

\begin{theorem}\label{thm: edge Omega X and face X}
Let $(X, x_0)$ be a (based) simplicial complex.  The function $\Phi$ from above induces an isomorphism of groups
$$\xymatrix{
F(X, x_0) \ar[r]^-{\cong}& E(\Omega X, \mathbf{x_0}).
}$$
\end{theorem}

\begin{proof}
First, we show that $\Phi$ induces a well-defined function 
\begin{equation}\label{eq: face group to edge group map}
    \Phi \colon F(X, x_0) \to E(\Omega X, \mathbf{x_0}).
\end{equation}
To this end, suppose a 
map $f \colon \left( I_{m} \times I_{n}, \partial (I_{m} \times I_{n}) \right) \to (X, x_0)$ represents $[f] \in F(X, x_0)$.    It is sufficient to show that, for contiguous maps $f \sim g$, and a trivial extension $\overline{f}$ of $f$, we have $\Phi(g)$ and $\Phi(\overline{f})$ each in the same extension-contiguity class as $\Phi(f)$ when considered as loops in $\Omega X$.

As above, we may view $f$ in terms of its sequence of rows
$\gamma_f  = ( \mathbf{x^m_0}, l^1, \ldots, l^{n-1}, \mathbf{x^m_0} )$, giving a loop of length $n$ in $\Omega X[m]$. Each row of $f$ gives a loop in $X$: we may write $l^j = (x_0, v^j_1, \ldots, v^j_{m-1}, x_0)$. Finally, viewing the loop $\Phi(f)$ as a map
$$\Phi(f) \colon I_{2m+n} \to \Omega X,$$
we have
$$\Phi(f)(t) = \begin{cases}
    \mathbf{x^{t+1}_0} & 0 \leq t \leq m-1 \\
    \mathbf{x^m_0} & t = m\\
    l^{t-m} & m+1 \leq t \leq m+n-1\\
    \mathbf{x^m_0} & t = m+n\\
    \mathbf{x^{(2m+n+1)-t}_0} & m+n+1 \leq t \leq 2m+n.
\end{cases}$$
For a contiguous $g \sim f$, we may write likewise $\gamma_g  = ( \mathbf{x^m_0}, k^1, \ldots, k^{n-1}, \mathbf{x^m_0} )$, with each row $k^j = x_0, w^j_1, \ldots, w^j_{m-1}, x_0$ and
$$\Phi(g)(t) = \begin{cases}
    \mathbf{x^{t+1}_0} & 0 \leq t \leq m-1 \\
    \mathbf{x^m_0} & t = m\\
    k^{t-m} & m+1 \leq t \leq m+n-1\\
    \mathbf{x^m_0} & t = m+n\\
    \mathbf{x^{(2m+n+1)-t}_0} & m+n+1 \leq t \leq 2m+n.
\end{cases}$$
To confirm that these give contiguous loops in $\Omega X$, we need to check that $\Phi(f)(\sigma) \cup \Phi(g)(\sigma)$ is a simplex of $\Omega X$ for each simplex $\sigma$ of $I_{2m+n+1}$.  That is, we want
$$\{ \Phi(f)(t), \Phi(f)(t+1), \Phi(g)(t), \Phi(g)(t+1) \}$$
to be a simplex of $\Omega X$, for each $0 \leq t \leq 2m+n$.  This is more-or-less tautologically satisfied for $0 \leq t \leq m-1$ and $m+n \leq t \leq 2m+n-1$, since in these ranges we have 

$\{ \Phi(f)(t), \Phi(f)(t+1), \Phi(g)(t), \Phi(g)(t+1) \} = $
$$ 
\begin{cases}
    \{ \mathbf{x^{t+1}_0}, \mathbf{x^{t+2}_0} \} & 0 \leq t \leq m-2 \\
    \mathbf{x^m_0} & t = m-1, m+n\\
   \{ \mathbf{x^{(2m+n+1)-t}_0}, \mathbf{x^{(2m+n)-t}_0} \}
     & m+n+1 \leq t \leq 2m+n-1.
\end{cases}
$$
Thus, we have $\Phi(f)(\sigma) \cup \Phi(g)(\sigma)$ is either an edge or a single vertex of $\Omega X$ here. For the sections of $\Phi(f)$ and $\Phi(g)$ that actually involve $f$ and $g$, we write $l^0 = \mathbf{x^m_0} = l^n$ and 
$k^0 = \mathbf{x^m_0} = k^n$, so that we have
$$\{ \Phi(f)(t), \Phi(f)(t+1), \Phi(g)(t), \Phi(g)(t+1) \} = \{ l^t, l^{t+1}, k^t, k^{t+1} \}$$
for $t = m, \dots, m+n-1$.

To check the simplex condition, write the four loops ``array style" as
\[
\left[ \begin{array}{c}
k^{t+1} \\
k^{t}  \\
l^{t+1}  \\
l^{t}  
\end{array}\right]
\qquad = \qquad
\left[\begin{array}{ccccc}
x_0  & w^{t+1}_1 &  \cdots &  w^{t+1}_{n-1} & x_0 \\
x_0  & w^{t}_1 &  \cdots &  w^{t}_{n-1} & x_0 
 \\
x_0  & v^{t+1}_1 &  \cdots &  v^{t+1}_{n-1} & x_0 \\
x_0  & v^{t}_1 &  \cdots &  v^{t}_{n-1} & x_0 
\end{array}\right].
\]  
The simplex condition we want satisfied is that the union of vertices from adjacent columns should form a simplex of $X$. But this condition is a direct translation of what the contiguity $f \sim g$ entails: we have $f(\sigma) \cup g(\sigma)$ a simplex of $X$ for the simplex
$$\sigma = \{ (i, t), (i+1, t), (i. t+1), (i+1, t+1) \}$$
of $I_m \times I_n$, and our notation means $f(i, j) = v^j_i$ and $g(i, j) = w^j_i$. It follows that $\Phi(f) \sim \Phi(g)$ as loops in $\Omega X$, for contiguous maps $f \sim g$.  

Next consider a trivial extension of $f$. It is sufficient to consider adding a single column to $f$ and adding a single row to $f$ separately, as these may be repeated in various combinations to obtain a general trivial extension of $f$. If we extend $f$ to $\overline{f}$ by adding a row, so that row-wise we may write
$$\gamma_{\overline{f}}  = ( \mathbf{x^m_0}, l^1, \ldots, l^{n-1}, \mathbf{x^m_0}, \mathbf{x^m_0} ),$$
then the effect on $\Phi(f)$ is to repeat the vertex $\mathbf{x^m_0}$, thus:
$$\Phi(\overline{f}) = S_0^{m} \cdot \gamma_f \cdot \mathbf{x^m_0} \cdot \widetilde{S_0^{m}}.$$
Repeating a vertex is one of the moves we have for operating within an extension contiguity class of the edge group, and so $\Phi(f)$ and $\Phi(\overline{f})$ will represent the same class in the edge group of $\Omega X$. On the other hand suppose we extend $f$ by adding a column.  Then we have  
$$\Phi(\overline{f}) = \gamma_0^{m+1} \cdot ( \mathbf{x^{m+1}_0}, \overline{l^1}, \ldots, \overline{l^{n-1}}, \mathbf{x^{{m+1}}_0} ) \cdot \widetilde{\gamma_0^{m+1}}.$$
Now parts (a) and (b) of \lemref{lem: 3-simplex extensions} give a contiguity of paths

$ (\mathbf{x^{m}_0}, \mathbf{x^{m+1}_0}, \mathbf{x^{m+1}_0}, \overline{l^1}, \ldots, \overline{l^{n-1}}, \mathbf{x^{{m+1}}_0}, \mathbf{x^{m+1}_0}, \mathbf{x^{m}_0} ) \sim $ 
$$( \mathbf{x^{m}_0}, \mathbf{x^{m}_0}, \mathbf{x^{m}_0}, l^1, \ldots, l^{n-1}, \mathbf{x^{{m}}_0}, \mathbf{x^{m}_0}, \mathbf{x^{m}_0} )$$
that extends to a contiguity of loops
$$\Phi(\overline{f}) \sim \gamma_0^{m} \cdot \mathbf{x^{m}_0} \cdot ( \mathbf{x^{m}_0}, l^1, \ldots, l^{n-1}, \mathbf{x^{{m}}_0} ) \cdot \mathbf{x^{m}_0} \cdot \widetilde{\gamma_0^{m}}.$$
Then removing the repeats of the vertex $\mathbf{x^{m}_0}$ shows that $\Phi(\overline{f})$ and $\Phi(f)$ will represent the same extension-contiguity class in the edge group of $\Omega X$ in this case too. This completes the argument that $\Phi$ induces a well-defined map of extension-contiguity classes as in (\ref{eq: face group to edge group map}). We abuse notation somewhat by setting $\Phi([f]) := [\Phi(f)]$ for a face sphere $f$ in $X$.  

The next step is to show that $\Phi$ induces a homomorphism.  Suppose that $[f]$ and $[g]$ are elements of $F(X, x_0)$ with 
$$f\colon \left(I_{m} \times I_{n}, \partial (I_{m} \times I_{n}) \right) \to (X,x_0)$$
and
$$g\colon \left(I_{r} \times I_{s}, \partial (I_{r} \times I_{s}) \right) \to (X,x_0).$$
Then we must show that $\Phi([f]\cdot [g])$ and $\Phi([f])\cdot \Phi([g])$ give the same element of $E(\Omega X, \mathbf{x_0})$.

Since $\Phi$ is well-defined on a contiguity-equivalence class, we may pre-process $[f]\cdot [g]$ as follows.  The map $f \cdot g$, when restricted to $[n+1, n+s+1] \times I_m$, agrees with $g \circ (\alpha^m_0\times \mathrm{id})$, which is contiguity equivalent to $g \circ (\alpha^m_r\times \mathrm{id})$ (as face spheres in $X$).  Since this contiguity equivalence is stationary on its bottom row, we may piece it together with the stationary contiguity on $I_{m+r+1} \times I_{n}$ to obtain a contiguity equivalence (again, of face spheres in $X$) $f \cdot g \simeq f*g,$ where $f*g$ denotes the map 
$$f*g \colon \left(I_{m+r+1} \times I_{n+s+1}, \partial (I_{m+r+1} \times I_{n+s+1}) \right) \to (X,x_0)$$
defined by
$$
f*g(i, j) := 
\begin{cases}
    f\circ(\alpha^{r+1}_m \times \mathrm{id}) & 0 \leq j \leq n\\
    g\circ(\alpha^{m+1}_r \times \mathrm{id}) & n+1 \leq j \leq n+r+1
\end{cases}
$$
Then $[f]\cdot [g]  = [f\cdot g] = [f*g]$, and we have 
$$\Phi([f]\cdot [g]) = [\Phi(f*g)].$$
Recall our notation from above, whereby we write $f$ and $g$ ``row-wise" as the loops $\gamma_f  = ( \mathbf{x^m_0}, l^1, \ldots, l^{n-1}, \mathbf{x^m_0} )$ and 
$\gamma_g  = ( \mathbf{x^r_0}, k^1, \ldots, k^{s-1}, \mathbf{x^r_0} )$ in $\Omega X$.  Then we have
$$\Phi(f*g) = \gamma_0^{m+r+1} \cdot \gamma_f\circ \alpha^{r+1}_m \cdot \gamma_g\circ \alpha^{m+1}_r \cdot \widetilde{\gamma_0^{m+r+1}},$$
where by $\gamma_f\circ \alpha^{r+1}_m$ we intend the loop in $\Omega X[m+r+1]$ 
$$( \mathbf{x^{m+r+1}_0}, l^1\circ \alpha^{r+1}_m, \ldots, l^{n-1}\circ \alpha^{r+1}_m, \mathbf{x^{m+r+1}_0} )$$
in $\Omega X$ and likewise for $\gamma_g\circ \alpha^{m+1}_r$.

On the other hand, we have
$$
\begin{aligned}
\Phi\left( [f]\right) \cdot \Phi\left( [g]\right) &= \Phi\left( [f\circ(\alpha^{r+1}_m \times \mathrm{id})]\right) \cdot \Phi\left( [g\circ(\alpha^{m+1}_r \times \mathrm{id})]\right)\\
&= \gamma_0^{m+r+1} \cdot \gamma_f\circ \alpha^{r+1}_m \cdot \widetilde{\gamma_0^{m+r+1}} \cdot \gamma_0^{m+r+1} \cdot\gamma_g\circ \alpha^{m+1}_r \cdot \widetilde{\gamma_0^{m+r+1}}.
\end{aligned}
$$
Now the middle section here is a concatenation of a path and its reverse, which we may replace up to a contiguity equivalence (relative its endoints) by the constant path at $\mathbf{x_0^{m+r+1}}$ of suitable length.  Then removing repeats of $\mathbf{x_0^{m+r+1}}$ from the central section results in a loop in $\Omega X$ that is a extension-contiguity equivalent to $\Phi(f*g)$ as displayed above.  It follows that we have $\Phi([f] \cdot [g]) = \Phi([f]) \cdot \Phi([g])$ in $E(\Omega X, \mathbf{x_0})$, and so $\Phi$ is a homomorphism.   

Surjectivity of $\Phi$ follows directly from \propref{prop: loop in X(M)}.  There, we showed that a typical element of $E(\Omega X, \mathbf{x_0})$ may be represented by an edge loop of the form
$$\widehat{\gamma}     = ( \mathbf{x_0}, \mathbf{x^1_0}, \ldots, \mathbf{x^M_0}, \ell^1, \ldots, \ell^p, \mathbf{x^M_0}, \ldots, \mathbf{x^1_0}, \mathbf{x_0} ),$$
with each vertex $\ell^j$ a loop of length $M$ in $\Omega X$. Up to extension contiguity of edge loops, we may repeat the vertex $\mathbf{x^M_0}$ at either end of the central section to obtain a representative
$$\gamma_0^M \cdot (  \mathbf{x^M_0}, \ell^1, \ldots, \ell^p, \mathbf{x^M_0} ) \cdot \widetilde{ \gamma_0^M}.$$
Then $( \mathbf{x^M_0}, \ell^1, \ldots, \ell^p, \mathbf{x^M_0} )$ may be viewed as $\gamma_f$ for a face sphere $f \colon I_M \times I_n \to X$, for which we have 
$$\Phi(f) = \gamma_0^M \cdot (  \mathbf{x^M_0}, \ell^1, \ldots, \ell^p, \mathbf{x^M_0} ) \cdot \widetilde{ \gamma_0^M}.$$
Thus $\Phi$ induces a surjection.

To show that the homomorphism induced by $\Phi$ is an injection, it is sufficient to show that a face sphere $f$ is extension-contiguity equivalent to the constant map at the basepoint, whenever the loop $\Phi(f)$ in $\Omega X$ is extension-contiguity equivalent to the constant loop in $\Omega X$.   

Suppose we have 
$$f\colon \left(I_{m} \times I_{n}, \partial (I_{m} \times I_{n}) \right) \to (X,x_0)$$
Then $\Phi(f)$ is a loop of length $2m+n+2$ in $\Omega X$.  By \lemref{lem: expand then contract}, we may assume that some trivial extension $\gamma$ of $\Phi(f)$ up to a suitable length $N$ is contiguity equivalent to the constant loop $\mathbf{x^N_0}$, i.e., we have
$$\gamma:= \Phi(f)\circ \alpha_{2m+n+2}^r \simeq \mathbf{x^N_0},$$
for $r = N-(2m+n+2)$. Suppose this contiguity equivalence is written out as
$$\gamma = \gamma_0  \sim \gamma_1 \sim \cdots \sim \gamma_r = \mathbf{x^N_0},$$
with each $\gamma_j$ a loop of length $N$ in $\Omega X$.  Let $M$ be the longest length of a loop in $X$ that occurs as a vertex in any of the loops $\gamma_j$, for $j=0, \ldots, r$. Then we ``same-size" to length $M$ every vertex of every loop $\gamma_j$ that occurs in the contiguity equivalence.
By part (b) of \lemref{lem: contiguity of extensions}, we have a contiguity equivalence of loops
$$[\gamma_0]_M  \sim [\gamma_1]_M \sim \cdots \sim [\gamma_r]_M,$$
each of which is a loop of length $N$  in $\Omega X[M]$.

If we use the notation of \lemref{lem: contiguity loops face spheres} to write 
$$\gamma_f = (\mathbf{x_0^m}, l^1, \ldots, \l^{n-1}, \mathbf{x_0^m}),$$
then the loop $[\gamma_0]_M$ is obtained from $[\gamma_f]_M$ by adding repeats of $\mathbf{x_0^M}$ at either end which, by part (d) of \propref{prop: contiguity results} 
 is contiguity equivalent to adding the suitable number of repeats of $\mathbf{x_0^M}$ at the end.  That is, we may extend the contiguity equivalence above to one of 
$$[\gamma_f]_M\circ \alpha_p^q \simeq [\gamma_0]_M  \simeq [\gamma_r]_M,$$
where $[\gamma_r]_M$ just consists of repeats of the constant loop at $x_0$ of length $M$.
Now \lemref{lem: contiguity loops face spheres} allows us to translate this contiguity equivalence into one of face spheres
$$f\circ (\alpha_m^q \times \alpha_n^p)  \simeq f_{ [\gamma_r]_M }.$$
As $[\gamma_r]_M$ just consists of repeats of the constant loop at $x_0$ of length $M$, the corresponding face sphere  $f_{ [\gamma_r]_M }$ is just the constant map at $x_0$ of $I_M \times I_N$.  Thus, we have an extension-contiguity equivalence of face spheres
between $f$ and a constant map.
\end{proof}

\begin{example}
In \cite{L-S-S} it is shown that we have $F(X, x_0) \cong \pi_2(|X|, x_0)$. Combining this with \thmref{thm: edge Omega X and face X}, we have isomorphisms
$$E(\Omega X, x_0) \cong F(X, x_0) \cong \pi_2( |X|, x_0).$$
For instance, let $X$ be the hollow $3$-simplex of \exref{ex: hollow 3-simplex}. Then the spatial realization $|X|$ is homotopy equivalent to $S^2$ and we have isomorphisms
$$E(\Omega X, x_0) \cong \pi_2( S^2, x_0) \cong \Z.$$
Because topologically we have $\pi_2( |X|, x_0) \cong \pi_1( \Omega |X|, x_0)$, it follows that for any $X$ we have an isomorphism
$$E(\Omega X, x_0) \cong \pi_1( \Omega |X|, x_0).$$
Then, because the edge group is isomorphic to the fundamental group of the spatial realization, our results so far yield an isomorphism of fundamental groups 
$$\pi_1( |\Omega X|, x_0) \cong \pi_1( \Omega |X|, x_0).$$
In the rest of the paper we improve on this result by showing a homotopy equivalence between  $|\Omega X|$ and $\Omega |X|$.
\end{example}

\section{Spatial Realization of $\Omega X$}\label{sec: spatial realization}

We hope eventually to establish a homotopy equivalence $|\Omega X| \simeq \Omega |X|$, where this latter denotes the topological based loop space of the (topological space) spatial realization $|X|$ of the simplicial complex $X$.  For the time being, we relate $|\Omega X|$ to Stone's approximations to the loop space from \cite{St79} (see section 4, in particular). 

For $X$ a simplicial complex, Stone gives a cell complex $N(k)$ for each $k$.  These $N(k)$ form a direct system by inclusion and Stone shows that the colimit is homotopy equivalent to the (topological) based loop space $\Omega |X|$.  Stone's point of view is that of Morse theory for a polyhedron (e.g. a triangulated manifold).     

\begin{definition}[Stone's $N(k)$]\label{def:stone}
Let $X$ be a simplicial complex. By a \emph{chain of length $k$} (of simplexes in $X$), we mean a sequence of simplices (of any dimensions) 
$$ \{ x_0 \}, \sigma_1, \ldots, \sigma_{k-1}, \{x_0\}$$
that satisfy the condition $\sigma_i \cup \sigma_{i+1}$ is a simplex of $X$, for $0 \leq i\leq k-1$.  Here, we interpret $\sigma_0$ and $\sigma_k$ to be the vertex $\{ x_0 \}$.  Then $N(k)$ is defined as a subset of points of the $(k-1)$-fold Cartesian product $|X| \times \cdots \times  |X|$ as
$$N(k) = \left\{ (x_1, \ldots, x_{k-1}) \mid x_i \in |\sigma_i| \text{ for }  \{ x_0 \}, \sigma_1, \ldots, \sigma_{k-1}, \{x_0\} \text{ a chain in $X$} \right\}$$
\end{definition}  

\begin{remark}\label{rmk:N-k-functorial}
The construction $N(k)$ is functorial in $X$.
Let $f : (X, x_{0}) \rightarrow (Y, y_{0})$ be a simplicial map.
For the sake of this discussion, denote by $N_{X}(k)$ the $k$th approximation for $X$, and similarly for $Y$.
Let $(\sigma_{1}, \ldots, \sigma_{k-1})$ be a chain of length $k$ in $X$, so that each $\sigma_{i}$ is a simplex
in $X$, and the union of each pair of adjacent simplices is also a simplex, as are $\{x_{0}\} \cup \sigma_{1}$ and $\{x_{0}\} \cup \sigma_{k-1}$.
Since $f$ is simplicial and direct images commute with unions, $(f(\sigma_{1}), \ldots, f(\sigma_{k-1}))$ is a chain of length $k$ in $Y$.
Therefore the iterated function, $|f|^{k-1} : | X |^{k-1} \rightarrow |Y|^{k-1}$, sends chains to chains, and so restricts and co-restricts
to a continuous function $N_{f}(k) : N_{X}(k) \rightarrow N_{Y}(k)$.

Clearly, this construction respects identities and composition.
\end{remark}

\begin{definition}[Our $\Omega X(k)$]
Let $X$ be a simplicial complex. By $\Omega X(k)$ we mean the subcomplex of $\Omega X$ using vertices that are edge loops of length up to $k$.
\end{definition}  

% Notice that, as we have defined them, both $N(k)$ and $|\Omega X(k)|$ are topological spaces. 

\begin{proposition}\label{prop: homotopy equivalent skeleta}
For each $k$, $N(k)$ is a deformation retract of $|\Omega X(k)|$.
\end{proposition}

To prove the proposition, we  will define continuous maps  $r \colon |\Omega X(k)| \to N(k)$ and $i \colon N(k) \to  |\Omega X(k)|$ and show that   $r\circ i = \mathrm{id}_{N(k)}$ and $i\circ r \approx \mathrm{id}_{|\Omega X(k)|}$.  

To define the map
$r \colon |\Omega X(k)| \to N(k),$
 we define it on each $| \sigma |$, for $\sigma$ a simplex of $\Omega X(k)$. Suppose we have
$\sigma = \{ l^0, \ldots, l^s \}$
for edge loops $l^j$ of length up to $k$.  Recall that if $l$ has length $m \leq k$, then $(l)_{k} = l \circ \alpha_{m}^{k-m}$ is the corresponding loop of length $k$ obtained by adding repeats of $x_{0}$ to the end. 
% Then
% %
% \[
% \left[ \begin{array}{c}
% (l^{s})_{k}  \\
% (l^{s-1})_{k}  \\
% \vdots\\
% (l^{1})_{k}  \\
% (l^{0})_{k}  
% \end{array}\right]
% \qquad = \qquad
% \left[\begin{array}{ccccc}
%  x_0  & v^s_1 &  \cdots &  v^s_{k-1} &  x_0 \\
%  x_0  & v^{s-1}_1 &  \cdots &  v^{s-1}_{k-1} &  x_0 \\
%  \vdots  & \vdots &  \cdots & \vdots &  \vdots \\
% x_0  & v^1_1 &  \cdots &  v^1_{k-1} &  x_0 \\
% x_0  & v^{0}_1 &  \cdots &  v^{0}_{k-1} &  x_0 \\
% \end{array}\right]
% \]
% %
% with each row an edge loop 
% %
% $$(l^j)_{k} \colon x_0 = v^j_0, v^j_1, \ldots, v^j_{k-1}, v^j_k = x_0.$$
% %
% Depending on whether or not $\overline{\gamma^j}$ needed to be extended, the last  entries $v^j_i$ for $i < k$ may be $x_0$ or not.
% Now our condition on $\sigma$ for it to be a simplex is that the union of the vertices in each pair of adjacent columns is a simplex of $X$, which entails that each column of vertices is a simplex of $X$.  Some of the vertices in a column may be repeats of each other. But we do not try to control for this at all, e.g. by eliminating repeats, as having some vertices repeated does not do any harm in our development. 
Suppose that we have a point $x \in |\sigma|$ (in the spatial realization $|\Omega X(k)|$).  
Using the barycentric coordinates of $x$ in $\sigma$, write
$x =  \sum_{j=0}^s\ b_j l^j,$
That is, we have $0 \leq b_j \leq 1$ for each $j$, and $\sum_{j=1}^s\ b_j = 1$.     
Then we may define $r(x) \in N(k)$ as 
\[
    r(x) = \left( \sum_{j=0}^s\ b_j (l^j)_{k}(1), \ldots, \sum_{j=0}^s\ b_j (l^j)_{k}(k-1) \right).
\]
For each $i = 1, \dots, k-1$, the vertices $\{  (l^0)_{k}(i), \ldots,  (l^s)_{k}(i) \}$ form a simplex $\sigma_i$ of $X$, and each coordinate $\sum_{j=0}^s\ b_j (l^j)_{k}(i)$ in the above expression gives a point in   $|\sigma_i| \subseteq X$.  Furthermore, we have $\sigma_i \cup \sigma_{i+1}$ a simplex of $X$ for each $i = 0, \ldots, k-1$.  So, we have 
$$ \{ x_0 \}, \sigma_1, \ldots, \sigma_{k-1}, \{x_0\}$$
a chain of length $k$ and  $f(x) \in N(k)$.   This map is evidently continuous on each simplex $\sigma$.  Furthermore, $f$ defined in this way on two simplices that overlap shares a common value on the overlap (same vertices and barycentric coordinates on the common face).  Hence, $f$ assembles into a continuous map of $|\Omega X(k)|$.  

Now we define a map $i \colon N(k) \to |\Omega X(k)|$.   Suppose we have 
$$y = \left( \sum_{j=0}^{s_1}\ b^1_j v^j_1, \sum_{j=0}^{s_2}\ b^2_j v^j_2, \ldots, \sum_{j=0}^{s_{k-1}}\ b^{k-1}_j v^j_{k-1} \right) \in N(k).$$
Here, we suppose $y \in \sigma_1 \times  \cdots \times \sigma_{k-1} \subseteq N(k)$, with each $\sigma_i$ an $s_i$-simplex of $X$ with vertices
$$\sigma_i = \{  v^0_i, \dots, v^{s_i}_i \}.$$
In each coordinate of $y$, we have barycentric coordinates that satisfy  $\sum_{j=0}^{s_i}\ b^i_j = 1$.  A key point in what follows is that when we multiply these sums, we have
\[
    1 = 1\times \cdots \times 1 = \left( \sum_{j=0}^{s_1}\ b^1_j \right)\times \cdots\times  \left( \sum_{j=0}^{s_{k-1}}\ b^{k-1}_j \right) 
    = \sum_{J \in \mathcal{J}} \ b^1_{j_1} \cdots  b^{k-1}_{j_{k-1}},
\]
where $\mathcal{J}$ is the set of all $(k-1)$-tuples $J = (j_1, \ldots, j_{k-1})$ with each $j_i \in \{ 0, \ldots, s_i \}$.  
This observation allows us to use the terms in the right-hand sum as a new set of barycentric coordinates.

For each $J \in \mathcal{J}$, 
let $l^{J} = l^{j_1, \ldots, j_{k-1}}$ denote the edge loop in $\Omega (k)$ 
with $l^{J}(i) = v_{i}^{j_{i}}$.

Now consider $\tau = \{ l^{J} \mid J \in \mathcal{J} \}$.
We claim that $\tau$ is a (rather high-dimensional) simplex in $\Omega X (k)$. 
Indeed, we have that $\{ l^{J}(i) \mid J \in \mathcal{J} \} = \sigma_{i}$, 
and so $\{ l^{J}(i) \} \cup \{ l^{J}(i+1) \} = \sigma_{i} \cup \sigma_{i+1}$, 
which by assumption is a simplex in $X$.

We may use the barycentric coordinates from above and define
\[
    i(y)= \sum_{J \in \mathcal{J}} \ b^1_{j_1} \cdots  b^{k-1}_{j_{k-1}} 
    l^{j_1, \ldots, j_{k-1}}.
\]
%
%\todo{Do we need to check this map is well-defined?  Perhaps say that $y =(y_1, \ldots, y_{k-1}) \in N(k)$ lies in some unique cell $\sigma_1 \times \cdots \sigma_{k-1}$ with each $\sigma_i$ the carrier of the coordinate $y_i \in X$.} 

\begin{proof}[Proof of Proposition~\ref{prop: homotopy equivalent skeleta}]
First, we prove that $r\circ i = \mathrm{id}\colon N(k) \to N(k)$.

As above, start with $y \in N(k)$ and obtain  
$$i(y)= \sum_J \ b^1_{j_1} \cdots  b^{k-1}_{j_{k-1}} l^{j_1, \ldots, j_{k-1}}.$$
The $i$th coordinate of $r(i(y))$ is then
\[
    z_i = \sum_{J} b_{j_{1}}^{1}  \cdots b^{k-1}_{j_{k-1}} l^{J}(i)
    = \sum_J \ b^1_{j_1} \cdots  b^{k-1}_{j_{k-1}} v^{j_i}_i
\]
in $|\sigma_i|$.  The sum is over all suitable $J$, but we may aggregate the terms that contain a fixed $v^{j_i}_i$, thus:
$$z_i = \left(\sum_{J, j_i = 0} \ b^1_{j_1} \cdots  b^{k-1}_{j_{k-1}}\right) v^{0}_i + \cdots + \left(\sum_{J, j_i = s_i} \ b^1_{j_1} \cdots  b^{k-1}_{j_{k-1}}\right) v^{s_i}_i.$$
For each $(k-1)$-tuple $J = (j_1, \ldots, j_{k-1})$, let $\widehat{J} = (j_1, \ldots, \widehat{j_i}, \ldots, j_{k-1})$, the $(k-2)$-tuple with the $i$th entry omitted. Then the first sum above may be written
$$\sum_{J, j_i = 0} \ b^1_{j_1} \cdots  b^{k-1}_{j_{k-1}} = \left( \sum_{\widehat{J}} \ b^1_{j_1} \cdots  b^{k-1}_{j_{k-1}}\right) b^0_i, $$
with the sum now over all $\widehat{J} = (j_1, \ldots, \widehat{j_i}, \ldots, j_{k-1})$.  But this sum is 
$$\sum_{\widehat{J}} \ b^1_{j_1} \cdots  b^{k-1}_{j_{k-1}} = \left( \sum_{j=0}^{s_1}\ b^1_j \right)\times \cdots \times \widehat{\left( \sum_{j=0}^{s_i}\ b^i_j \right)} \times \cdots \times  \left( \sum_{j=0}^{s_{k-1}}\ b^{k-1}_j \right) = 1 \times \cdots \times 1 = 1,$$
since the $b^i_j$ are barycentric coordinates for each $i$.  Likewise for the other sums involved in the expression above for $z_i$, and it follows that we have 
$$z_i =  b_{0}^{i} v^{0}_i + \cdots + b_{s_i}^{i}  v^{s_i}_i$$ 
and thus $z_i = y_i$, or $r\circ i(y) = y$.

Next, we show that $i\circ r \simeq \mathrm{id} \colon |\Omega X(k)| \to |\Omega X(k)|$.

Suppose $x \in |\Omega X(k)|$.  As above, where we defined the map $r$, we may write 
$x =  \sum_{j=0}^s\ b_j l^j,$
with the $l^j$ the vertices of a simplex $\sigma$  in $\Omega X(k)$.  Tracking the definitions of $r$ and $i$ from above, we may see that $i\circ r(x)$ is contained in a (much larger) simplex $\Sigma$ of
$\Omega X(k)$ that contains the simplex $\sigma$.  Then $x$ and $i\circ r(x)$ are contained in (the spatial realization of) a simplex of $\Omega X(k)$. As in 1.7.4 and 1.7.5 of \cite{HiWi60}, this implies that 
$i\circ r$ and $\mathrm{id}$ are homotopic, using the straight line homotopy pointwise.
  
\end{proof}

In fact, the homotopy equivalence established above is natural, as expressed in the following two
propositions.

\begin{proposition}\label{prop:r-natural}
Let $f : (X, x_{0}) \rightarrow (Y, y_{0})$ be a simplicial map.
For each $k$, the diagram
\[
	\begin{tikzcd}
		{|\Omega X(k)|} \arrow[r, "r_{X}"]
			\arrow[d, "|\Omega f|"']
			& N_{X}(k) \arrow[d, "N_{f}(k)"] \\
		{|\Omega Y(k)|} \arrow[r, "r_{Y}"']
			& N_{Y}(k)
	\end{tikzcd}
\]
commutes.
\end{proposition}

\begin{proof}
Let $y \in | \Omega X (k) |$ have carrier the $m$-simplex $\sigma = \{ l^{0}, \ldots, l^{m} \}$.
Write $y = \sum_{j=0}^{m} b_{j} l^{j}$;  then 
\[
	N_{f}(k)(r_{X}(y))_{i}
        = N_{f}(k)\left( \sum_{j=0}^{m} b_{j}l^{j}(i)\right) 
						= \sum_{j=0}^{m} b_{j}f(l^{j}(i)).
\]
Meanwhile, 
\begin{multline*}
	r_{Y}(|\Omega f|(y))_{i} 
        = r_{Y}\left(|\Omega f|\left(\sum_{j=0}^{m} b_{j} l^{j}\right)\right)_{i}
		= r_{Y}\left(\sum_{j=0}^{m} b_{j} (\Omega f)(l^{j})\right)_{i} \\
		= r_{Y}\left(\sum_{j=0}^{m} b_{j} (f \circ l^{j})\right)_{i} 
				= \sum_{j=0}^{m} b_{j} (f \circ l^{j})(i) 
\end{multline*} 
and so $N_{f}(k)(r_{X}(y)) = N_{f}(k)\left(r_{X}(y)\right)$ as desired.
\end{proof}

\begin{proposition}\label{prop:i-natural}
Let $f : (X, x_{0}) \rightarrow (Y, y_{0})$ be a simplicial map.
For each $k$, the diagram
\[
	\begin{tikzcd}
		N_{X}(k) \arrow[d, "N_{f}(k)"'] 
			\arrow[r, "i_{X}"]
			& {|\Omega X(k)|} \arrow[d, "|\Omega f|"] \\
		N_{Y}(k) \arrow[r, "i_{Y}"']
			& {|\Omega Y(k)|}
	\end{tikzcd}
\]
commutes.
\end{proposition}

\begin{proof}
Suppose $y = (y_{1}, \ldots, y_{k-1}) \in N_{X}(k)$.
As in the construction (Section [...]), suppose that the carrier of $y_{i}$ is $\sigma_{i} = \{ v_{i}^{0}, \ldots, v_{i}^{s_{i}} \}$,
and write $y_{i} = \sum_{j=0}^{s_{i}} b_{j}^{i} v_{i}^{j}$.
For each sequence $J = (j_{1}, \ldots, j_{k-1})$ with $j_{i} \in \{ 0, \ldots, s_{i} \}$, we have two associated edge paths, one in 
$X$, corresponding to $y$, and the other in $Y$, corresponding to $f(y)$.  
The first we denote by $l^{J}$, with $l^{J}(i) = v_{i}^{j_{i}}$ (of course, $l^{J}(0) = l^{J}(k) = x_{0}$).
The second we will denote by $\lambda^{J}$.  
Since $f(y_{i}) = \sum_{j=0}^{s_{i}} b_{j}^{i} f(v_{i}^{j})$, we have for all $i$ that $\lambda^{J}(i) = f(v_{i}^{j_{i}})$.
It follows that $\lambda^{J} = f \circ l^{J}$.

Thus
\[
	|\Omega f|(i_{X}(y)) = |\Omega f| \left( \sum_{J} b_{J}l^{J} \right)
		= \sum_{J} b_{J} (f \circ l^{J}) 
		= i_{Y}(f(y))
		= i_{Y}(N_{f}(k)(y)),
\]
as desired.
\end{proof}

\section{Homotopy Direct Limits}\label{sec: limits}

In this section, we show that $|\Omega X|$ is homotopy equivalent to the homotopy colimit of
Stone's approximations $N(k)$ (Definition~\ref{def:stone}).  
We first recall the direct system for these approximations.

Let $k_{1}, k_{2}, \ldots$ be a sequence of positive integers in which each term divides the following term.
For each $k$ in the sequence, fix a partition of $[0,1]$, 
$T_{k} : 0 = t_{0} < t_{1} < \cdots < t_{k-1} < t_{k}= 1$.
Suppose that $k'=ck$ for some positive integer $c$, and write $T_{k'}$ as
$0 = t'_{0} < \cdots < t_{k'-1} < t_{k'} =  1$.
We require that $t'_{ci} = t_{i}$,
and that the mesh of $T_{k}$ approaches $0$ as $k$ increases.

For successive terms $k$, $k'$ in the above sequence, Stone defines a Hurewicz cofibration 
$\iota_{k',k} : N(k) \rightarrow N(k')$  as follows.   
Let $b \in N(k)$;  then for some chain $(\sigma_{i})$ of length $k$ in $X$, we can write $b = (b_{i})$, 
where $b_{i} \in |\sigma_{i}|$ for $i = 1, \ldots, k-1$.
Suppose that $k' = ck$ for some positive integer $c$. 
For a given $i$ satisfying $1 \leq i \leq k$, suppose that $c(i-1) \leq \ell < ci$.
Define $b'_{\ell}$ to be the following convex combination of $b_{i-1}$ and $b_{i}$ in 
$|\sigma_{i-1} \cup \sigma_{i}|$:
\[
    b'_{\ell} = \frac{t_{i}-t'_{\ell}}{t_{i} - t_{i-1}} b_{i-1} 
    + \frac{t'_{\ell} - t_{i-1}}{t_{i}-t_{i-1}} b_{i},
\]
where we use the convention that $b_{0} = b_{k} = x_{0}$.
Define $\iota_{k',k}(b) = (b'_{\ell})$.
In~\cite{St79}, Stone shows that $\Omega |X|$ is homotopy equivalent to the homotopy colimit of the maps $\iota_{k',k}$.

Figure~\ref{fig:iota} illustrates $\iota_{6,3}(b)$, where $(b) = (b_{1},b_{2}) \in \sigma_{1} \times \sigma_{2} \in N(3)$.
In this case, $\iota_{6,3}(b) = (b_{1}', \ldots, b_{5}')$, where $b_{1}' \in \langle x_{0}, \sigma_{1}\rangle$, $b_{2}' = b_{1}$,  $b_{3}' \in \langle \sigma_{1}, \sigma_{2} \rangle$, $b_{4}' = b_{2}$, and $b_{5}' \in \langle \sigma_{2}, x_{0} \rangle$.

Let $j_{k',k} : \Omega X(k) \rightarrow \Omega X(k')$ be the inclusion.

\begin{proposition}\label{prop:htpy-commute}
The diagram
\[
	\begin{tikzcd}
		{| \Omega X(k) |} \arrow[r, "|j_{k',k}|"]
			\arrow[d, "r_{k}"']
			& {| \Omega X(k') |} \arrow[d, "r_{k'}"] \\
		N(k) \arrow[r, "\iota_{k',k}"]
			& N(k')
	\end{tikzcd}
\]
commutes up to homotopy.
\end{proposition}

\begin{proof}
Recall that $N(k')$ is a polyhedral subcomplex of $X^{k'-1}$.  
Let $p_{i} : N(k') \rightarrow X$ be the restriction of the $i$th projection map, for $i = 1, \ldots, k'-1$.
It suffices to show that $p_{i} \circ \iota_{k',k} \circ r_{k} \simeq p_{i} \circ  r_{k'} \circ |j_{k',k}| $ 
for each $i$.

Suppose that $k' = ck$ for some  integer $c \geq 2$.
Let $y \in |\Omega X(k)|$.
Suppose that the carrier of $y$ is the simplex $\{ l^{0}, \ldots, l^{n}\}$.
Using barycentric coordinates, $y = \sum_{\ell=0}^{n} b_{\ell} l^{\ell}$.
Write $l^{\ell}(i) = v_{i}^{\ell}$ (so $v_{0}^{\ell} = v_{k}^{\ell})$.
As usual, set $\sigma_{i} = \{ v_{i}^{0}, \ldots, v_{i}^{\ell} \}$, for $i = 1, \ldots, k-1$.
Using the barycentric coordinates for $y$, set 
$y_{i} = \sum_{\ell=0}^{n} b_{\ell} v_{i}^{\ell} \in |\sigma_{i}|$.

Recall that we have fixed $k$- and $k'$-partitions $T_{k}$ and $T_{k'}$ of the unit interval $I$, such that 
$t'_{ci} = t_{i}$ for $i = 0, \ldots, k$.
From the definitions, $r_{k'}\circ |j_{k',k} |(y) = (y_{1}, \ldots, y_{k-1}, x_{0}, \ldots, x_{0})$.
Furthermore, $\iota_{k',k} \circ r_{k} (y) = (z_{1}, \ldots, z_{k'-1})$, where $z_{ci} = y_{i}$ and $z_{\ell}$ lies on the line segment
from $y_{i-1}$ to $y_{i}$ if $c(i-1) < \ell < ci$.
Therefore there is a piecewise-linear path $\alpha_{y} : I \rightarrow |X|$ in the
chain $(\sigma_{1}, \ldots, \sigma_{k-1})$ such that $\alpha_{y}(0) = \alpha_{y}(1) = x_{0}$, and
for $i = 1, \ldots, k'-1$, $\alpha_{y}(t'_{i})=z_{i}$.
Furthermore, $\alpha_{y}$ depends continuously on $y$; that is, the evaluation map $|\Omega X(k)| \times I \rightarrow |X|$,
$(y,s) \mapsto \alpha_{y}(s)$, is continuous.

Define a homotopy $h_{i} : |\Omega X(k)| \rightarrow X$ by 
%from $\iota_{k',k} \circ r_{k}$  to $r_{k'} \circ |j_{k', k}|$ by
\[
		h_{i}(y,s) = \left\{
			\begin{array}{ll}
				\alpha_{y}(s(t_{i}-t_{i}')+t_{i}')		& \text{if }i=1,\ldots,k-1	\\
				\alpha_{y}(s(1-t_{i}')+t_{i}')			& \text{if }i=k, \ldots, k'-1
			\end{array}
			\right.
\]
for $i = 1, \ldots, k'-1$.

For all $i$, we have 
\[
    h_{i}(y,0) = \alpha_{y}(t'_{i}) = z_{i} = p_{i} \circ \iota_{k',k} \circ r_{k}(y).
\]
If $i \leq k-1$, then 
\[
    h_{i}(y,1) = \alpha_{y}(t_{i}) = \alpha_{y}(t'_{ci}) 
    = z_{ci}= y_{i} = p_{i} \circ r_{k'}\circ |j_{k',k} |(y).
\]
Otherwise, if $i \geq k$, then
\[
    h_{i}(y,1) = \alpha_{y}(1) = x_{0} = p_{i} \circ r_{k'}\circ |j_{k',k} |(y).
\]
Therefore the maps $h_{i}$ determine a homotopy 
$h : \iota_{k',k} \circ r_{k} \simeq r_{k'} \circ |j_{k', k}|$.
\end{proof}

It is a standard exercise that the homotopy-commutative square in Proposition~\ref{prop:htpy-commute}
can be ``rigidified'', as expressed in the following proposition.

\begin{proposition}\label{prop:strict-commute}
For all $k$ in the sequence $(k_{1}, k_{2}, \ldots)$, there exists a map $\bar{r}_{k} : |\Omega X(k)| \rightarrow N(k)$ that is homotopic
to $r_{k}$, such that the diagram
\[
	\begin{tikzcd}
		{| \Omega X(k) |} \arrow[r, "|j_{k',k}|"]
			\arrow[d, "\bar{r}_{k}"']
			& {| \Omega X(k') |} \arrow[d, "\bar{r}_{k'}"] \\
		N(k) \arrow[r, "\iota_{k',k}"]
			& N(k')
	\end{tikzcd}
\]
commutes strictly.
\end{proposition}

\begin{proof}
Let $\bar{r}_{k_{1}} = r_{k_{1}}$.
For $i \geq 1$, suppose that $\bar{r}_{k_{i}}$ has been constructed.
Since $j_{k_{i+1},k_{i}} : |\Omega X(k_{i})| \rightarrow |\Omega X(k_{i+1})|$ is a countable relative cell extension, it is a Hurewicz
cofibration.
Thus we may replace $r_{k_{i+1}}$ with a homotopic map that makes the diagram commute.
\end{proof}

Since each inclusion $|j_{k',k}|$ is a Hurewicz cofibration, and the sequence $(k_{i})$ is cofinal in the
natural numbers, we conclude that $|\Omega X| = \operatorname{hocolim}_{k}|\Omega X(k)|$.
Taking colimits, we obtain a map $\bar{r} : |\Omega X| \rightarrow \operatorname{hocolim}_{k}N(k)$.

\begin{theorem}\label{thm:r-equiv}
The map $\bar{r} : |\Omega X| \rightarrow \operatorname{hocolim}_{k}N(k)$ is a homotopy equivalence.
\end{theorem}

\begin{proof}
By construction, for each $k$ in the sequence $(k_{i})$, the map $\bar{r}$ restricts to 
$\bar{r}_{k} : |\Omega X (k)| \rightarrow N(k)$.  
Since $\bar{r}_{k} \simeq r_{k}$ and $r_{k}$ is a homotopy equivalence, so is $\bar{r}_{k}$.
By Milnor~\cite[Theorem A]{milnor63}, $\bar{r}$ is itself a homotopy equivalence.
\end{proof}

Combining Theorem~\ref{thm:r-equiv} with Stone's result [ref], we obtain the following corollary.

\begin{corollary}\label{cor:r-equiv}
    There is a homotopy equivalence, $|\Omega X| \simeq \Omega |X|$.
\end{corollary}

\begin{figure}\label{fig:iota}
    \begin{center}
        \begin{tikzpicture}
        	%
        	% domain
        	%
        	\coordinate [label=left:$x_{0}$] (P) at (0,-1);
        	\coordinate [label=right:$x_{0}$] (Q) at (12,-1);
        	\coordinate (A) at (2,3);
        	\coordinate (B) at (4,-0.5);
        	\coordinate (C) at (8,-2);
        	\coordinate (D) at (9,3);
        	\fill (P) circle (2pt);
        	\fill (Q) circle (2pt);
        	\draw (A) --  (B);
        	\draw[dotted] (B) -- (Q);
        	\filldraw[fill=gray, draw=black] 
        		(B) -- (C) -- (D) -- cycle;
        	\draw[dotted] (A) -- (P) -- (B);
        	\draw[dotted] (C) -- (A) -- (D) -- (Q) -- cycle;
        	\coordinate [label = left:$b_{1}$] (x1) at ($(A)!0.3!(B)$);
        	\fill (x1) circle (2pt);
        	\coordinate (temp) at ($(B)!0.3!(C)$);
        	\coordinate [label=above right:$b_{2}$] (x2) at ($(temp)!0.25!(D)$);
        	\fill (x2) circle (2pt);
        	\draw[dashed, color=red] (P) -- (x1) -- (x2) -- (Q);
        	%
        	% codomain
        	%
        	\draw [->] (6,-2) -- node [anchor = west] {$\iota_{6,3}$} (6,-3.5);
        	\coordinate [label=left:$x_{0}$] (PP) at (0,-8);
        	\coordinate [label=right:$x_{0}$] (QQ) at (12,-8);
        	\coordinate (AA) at (2,-4);
        	\coordinate (BB) at (4,-7.5);
        	\coordinate (CC) at (8,-9);
        	\coordinate (DD) at (9,-4);
        	\fill (PP) circle (2pt);
        	\fill (QQ) circle (2pt);
        	\draw (AA) --  (BB);
        	\draw[dotted] (BB) -- (QQ);
        	\filldraw[fill=gray, draw=black] 
        		(BB) -- (CC) -- (DD) -- cycle;
        	\draw[dotted] (AA) -- (PP) -- (BB);
        	\draw[dotted] (CC) -- (AA) -- (DD) -- (QQ) -- cycle;
        	\coordinate [label = left:$b_{2}'$] (b2prime) at ($(AA)!0.3!(BB)$);
        	\coordinate (temp2) at ($(BB)!0.3!(CC)$);
        	\coordinate [label=above right:$b_{4}'$] (b4prime) at ($(temp2)!0.25!(DD)$);
        	\draw[dashed, color=red] (PP) -- (b2prime) -- (b4prime) -- (QQ);
        	\coordinate [label=above:$b_{1}'$] (b1prime) at ($(PP)!0.3!(b2prime)$);
        	\coordinate [label=above:$b_{3}'$] (b3prime) at ($(b2prime)!0.2!(b4prime)$);
        	\coordinate [label=above:$b_{5}'$] (b5prime) at ($(b4prime)!0.7!(QQ)$);
        	\foreach \x in {(b1prime), (b2prime), (b3prime), (b4prime), (b5prime)}
        		\fill \x circle (2pt);
        \end{tikzpicture}
    \end{center}
    \caption{An illustration of $\iota_{6,3} : N(3) \rightarrow N(6)$}
\end{figure}
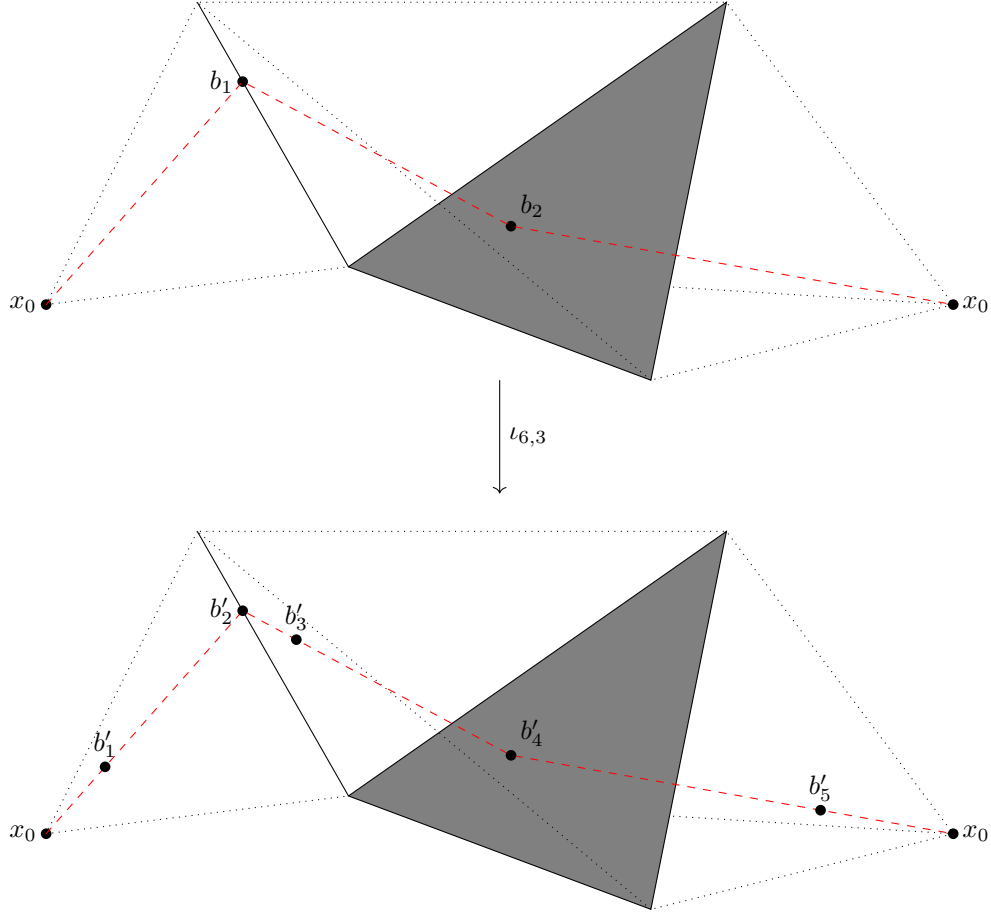

\section{Future Work}\label{sec: future}

There are a number of ideas that flow from our results that we intend to develop in subsequent articles. We indicate them briefly here.

Firstly, although we have focussed on the based loop space here, the construction of $\Omega X$ given in \secref{sec: Omega X} 
readily adapts or extends to similar constructions for various path and loop spaces, such as the based path space $\mathcal{P}X$, the free path space $PX$ and the free loop space $\Lambda X$. The resulting constructions are  variants of the path complexes described in \cite{FGMV21}, as $\Omega X$ is a variant of the $\widetilde{\Omega}X$ of \cite{Gr02} (see the discussion in \secref{sec: Intro}). We will investigate whether a somewhat different formulation of path and loop spaces may allow for advances on the topics considered in \cite{FGMV21}, such as category and topological complexity. 

The edge group of a (finite) simplicial complex has appeal from an algorithmic point of view. Although $\Omega X$ is not finite, it does have the ``locally finite" property mentioned in \secref{sec: Intro} (each vertex is of finite valency). Since $E(\Omega X, x_0) \cong F(X, x_0)$, it should be possible to approach $F(X, x_0)$  (and higher homotopy groups, suitably defined) in an similarly algorithmic way.  In particular, we hope to incorporate some of the ideas from \cite{Fr-Mu99} to develop ways of applying our $F(X, x_0)$ to analyze features of 3D digital images that are detectable by (digital counterparts to) second homotopy groups.  Such an approach using fundamental groups is discussed in \cite{Gr02}.

\bibliographystyle{amsplain}
%\bibliography{FaceGroup}

\providecommand{\bysame}{\leavevmode\hbox to3em{\hrulefill}\thinspace}
\providecommand{\MR}{\relax\ifhmode\unskip\space\fi MR }
\providecommand{\MRhref}[2]{%
  \href{http://www.ams.org/mathscinet-getitem?mr=#1}{#2}
}
\providecommand{\href}[2]{#2}

\end{document}